\documentclass[12pt]{article}
\usepackage{amsmath,amssymb,amsfonts,amsthm,bbm,bm}
\usepackage{color}
\usepackage{graphicx}
\usepackage{booktabs,mathtools,stmaryrd}
\usepackage{multirow}

\newtheorem  {proposition} {\hspace{15pt}Proposition}[section]
\newtheorem  {theorem}    {\hspace{15pt} Theorem}[section]

\newtheorem  {lemma}     {\hspace{15pt} Lemma}[section]

\def\uN{\mathbbm{u}}
\def\pN{\mathbbm{p}}
\def\qN{\mathbbm{q}}
\def\wN{\mathbbm{w}}
\def\wph{\mathbbm{W}}
\def\uph{\mathbbm{U}}
\def\pph{\mathbbm{P}}
\def\qph{\mathbbm{Q}}
\def\vph{\mathbbm{z}}
\def\vphu{\mathbbm{v}}
\def\vphq{\mathbbm{r}}
\def\vphp{\mathbbm{s}}
\def\Vph{\mathbbm{Z}}
\def\Vphu{\mathbbm{V}}
\def\Vphq{\mathbbm{R}}
\newcommand{\dual}[2]{\left\langle#1,#2\right\rangle}

\newcommand{\norm}   [1] {\left\Vert#1\right\Vert}

\newcommand{\enorm }  [1] {\interleave #1\interleave_{E}}

\newcommand{\jump}   [1] {[\![#1]\!]}
\def\af{a^\textsf{1}}
\def\as{a^\textsf{2}}

\begin{document}
	
	\title{\bf The local discontinuous Galerkin method on layer-adapted meshes
		for  time-dependent singularly perturbed convection-diffusion problems}
	\author{Yao Cheng\thanks{School of Mathematical Sciences, Suzhou University of Science and Technology, P.~R.~China.\  email: ycheng@usts.edu.cn.} 
		\and Yanjie Mei\thanks{International Education School, Suzhou University of Science and Technology, P.~R.~China.\  email: yjmei@post.usts.edu.cn.}
	\and Hans-G\"org Roos\thanks{Institute of Numerical Mathematics,
	Technical University of Dresden, Germany.\  email: hans-goerg.roos@tu-dresden.de.}
}

	\maketitle
	
	\begin{abstract}

In this paper we analyze the error as well for the semi-discretization as the full
discretization of a time-dependent convection-diffusion problem. We use for the discretization
in space the local discontinuous Galerkin (LDG) method on a class of layer-adapted meshes
including Shishkin-type and Bakhvalov-type meshes and the implicit $\theta$-scheme in time.
For piecewise tensor-product polynomials of degree $k$
we obtain uniform or almost uniform error estimates with respect to space of
order $k+1/2$ in some energy norm and optimal error estimates with respect to time.
Our analysis is based on careful approximation error estimates for the Ritz projection related to the
stationary problem on the anisotropic meshes used.
We discuss also improved estimates in the one-dimensional case
and the use of a discontinuous Galekin discretization in time.
Numerical experiments are given to support our theoretical results.

\


\end{abstract}

\section{Introduction}

We consider the two-dimensional singularly perturbed problem
\begin{subequations}\label{cd:spp:2d}
\begin{align}
 u_t-\varepsilon \Delta u + \bm a \cdot \nabla u + b u
&= f,  & \textrm{in}& \;  \Omega\times (0,T],
\\
 u|_{t=0}&=u_0,& \textrm{in}&\; \overline{\Omega},
\\
 u|_{\partial\Omega} &= 0, & \textrm{for}&\; t\in (0,T],
\end{align}
\end{subequations}
where $\Omega=(0,1)^2$, $T>0$, $\bm a=\bm a(x,y),b=b(x,y),f=f(x,y,t),u_0=u_0(x,y)$
are sufficiently smooth and
\begin{equation}\label{assumption:coef:2d}
\bm a=(\af(x,y),\as(x,y))
\geq(\alpha_1,\alpha_2),\quad
b- \frac12 \nabla\cdot \bm a \geq \beta,
\quad (x,y)\in \Omega,
\end{equation}
for some positive constants $\alpha_1,\alpha_2$ and $\beta$.

Such singularly perturbed problems arise in many applications.
They serve as a model problem
which helps to understand the behavior of numerical methods
in the presence of layers for more complex problems
like the linearized Navier-Stokes equations with high Reynolds number,
the drift-diffusion equations of semiconductor device modeling,
and the Black-Scholes equation from financial modeling \cite{Roos2008}.

The solution of such problems
usually exhibits rapid changes along the boundary,
i.e., the solution forms boundary layers.
 The use of a standard Galerkin finite element method
yields unsatisfactory numerical solution in that situation.
Even on an adapted mesh small oscillations occur, and it can be
difficult to solve the associated discrete algebraic system
efficiently \cite{Linss2001}.

Therefore,
many stabilised finite element methods were developed.
Several such methods are described in \cite{Roos2008}.
The methods described in \cite{Roos2008},
for instance,
streamline-diffusion finite element method (SDFEM),
local projection stabilization (LPS),
continuous interior penalty (CIP)
and discontinuous Galerkin (DG) methods
like nonsymmetric interior penalty DG (NIPG) and symmetric interior penalty DG (SIPG)
use the primal weak formulation for the stationary problem.
Instead, we use the flux formulation, especially
the local discontinous Galerkin (LDG) method.

The LDG method is a kind of finite element methods
which was first proposed
as a generalization of the discontinuous Galerkin method
for a  convection-diffusion problem \cite{Cockburn1998}.
This method inherits many advantages of the DG method
and can simulate the acute change
of singular solution very well \cite{Xu2010}.
The  LDG method has also already been used to solve singularly perturbed problems.
In \cite{Xie2009}, Xie et al. demonstrated that
the LDG solution does not produce any oscillations
on a uniform mesh.
In \cite{Cheng2015,Cheng2017},
Cheng et al. realized  a local stability analysis and proved double-optimal
local error estimates for two explicit fully-discrete LDG methods
on quasi-uniform meshes.
In \cite{Zhu2014},
Zhu et al. performed a uniform convergence analysis
of the LDG method on a standard Shishkin mesh for a stationary problem.

However, as far as the authors' know,
there exist  no results concerning the uniform convergence
of the fully-discrete LDG method on general layer-adapted meshes.
Our paper is characterized by the following features:

\begin{enumerate}
  \item
  In the literature, often Shishkin meshes are used for the discretization of singularly
  perturbed problems due to the simple structure of the mesh.
  However, this mesh gives non-optimal error estimates up to a logarithmic factor
which gets more and more influence for an increasing degree of the polynomials used.
We consider meshes defined by an adequately chosen mesh-generating function in the
layer region which are uniform outside the layer region. Thus we allow as well Shishkin-type (S-type) meshes
as Bakhvalov-type (B-type) meshes, introduced in \cite{Roos1999}. Our unified analysis yields
optimal error estimates for S-type meshes, the results for B-type meshes are almost
optimal.

\item
The key ingredient in our analysis is the use of the Ritz projection related to the stationary problem.
The approximation error estimates obtained on our class of
anisotropic meshes are new.
In the analysis, three two-dimensional local Gauss-Radau projections as well as their stability
and approximation property play an important role.
It is worthy to point out that a modified superconvergence property
and a stability result of \cite{Zhu2014} are particularly used
to derive the error estimates
on the considered class of meshes.

\item
For simplicity, we use an implicit $\theta$-scheme for the discretization in time
and establish (almost) optimal error estimates for the fully discrete LDG-scheme.
In the literature, also BDF methods (backward differencing) \cite{Dolejsi2010} and
Runge-Kutta methods \cite{Vlasak2017} are used.
Since the spatial mesh size is extremely small in the boundary layer region,
explicit time discretizations in \cite{Cheng2015,Cheng2017} are not suitable.
Alternatively, one can also use a DG method for the discretization in time,
as described in the Appendix (see also \cite{Franz2018}, where general stabilization
methods for the primal weak formulation are combined with DG in time).

\end{enumerate}

The paper is organized as follows.
In Section \ref{sec:LDG}, we present the LDG method
for the given singularly perturbed problem \eqref{cd:spp:2d}.
Layer-adapted meshes are introduced in Section \ref{sec:meshes}.
In Section \ref{sec:projections}, we define a Ritz projection
and estimate its approximation errors.
Based on this projection, we realize  error estimates as well
for the semi-discrete as for the fully-discrete LDG method
in Section \ref{sec:convergence:energy:norm}.
In Section \ref{sec:experiments} we present some numerical experiments.
In the last section,
we sketch improved results for the one-dimensional case
and the DG  discretization in time.

\section{The LDG method}
\label{sec:LDG}
\setcounter{equation}{0}

In this section we present the
 semi-discrete and the fully-discrete LDG method for \eqref{cd:spp:2d}.

Let $\Omega_N=\{K_{ij}\}_{i=1,2,\dots,N_x}^{j=1,2,\dots,N_y}$
be a rectangular tessellation of $\Omega$ with elements\\
$K_{ij}=I_i\times J_j$,
where $I_i=(x_{i-1},x_i)$ and $J_j=(y_{j-1},y_j)$.
Denote by $h_{x,i}=x_i-x_{i-1}$, $h_{y,j}=y_j-y_{j-1}$,
$h_{ij}=\max\{h_{x,i},h_{y,j}\}$, $\hbar_{ij}=\min\{h_{x,i},h_{y,j}\}$,
$h=\max_{K_{ij}\in\Omega_N} h_{ij}$ and $\hbar=\min_{K_{ij}\in\Omega_N} \hbar_{ij}$.
Define the discontinuous finite element space as
\begin{equation}
\mathcal{W}_{N}= \{v\in L^2 (\Omega)\colon
v|_{K} \in \mathcal{Q}^{k} (K),  K\in\Omega_N \},
\end{equation}
where $\mathcal{Q}^k(K)$ denotes the space of the tensor-product polynomials
of degree  $k$ in each variable on $K$.
$\mathcal{W}_{N}$ is contained in the broken Sobolev space
\begin{equation}
\mathcal{H}^1(\Omega_N)=\{v\in L^2(\Omega):
v|_{K}\in H^1(K), \ K\in\Omega_N\},
\end{equation}
whose functions allow discontinuity across element interfaces.
For $v\in \mathcal{H}^1(\Omega_N)$ and $y\in J_j$, $j=1,2,\dots,N_y$,
we use $v^\pm_{i,y}=\lim_{x\to x_{i}^\pm}v(x,y)$
to express the traces evaluated from
the right element $K_{i+1,j}$ and the left element $K_{ij}$.
The jump on the vertical edge is thus denoted by
$\jump{v}_{i,y}= v^{+}_{i,y}-v^{-}_{i,y}$
for $i=1,2,\dots,N_x-1$, $\jump{v}_{0,y}=v^{+}_{0,y}$
and $\jump{v}_{N_x,y}=-v^{-}_{N_x,y}$.
Analogously, for $x\in I_i$ and $i=1,2,\dots,N_x$, we can define
the traces $v^\pm_{x,j}$ and the jumps $\jump{v}_{x,j}$, $j=0,1,\cdots,N_y$
on the horizontal edges.

Rewrite \eqref{cd:spp:2d} into an equivalent first order system
\begin{align}\label{para:rewrite:2d}
u_t-p_x-q_y+ \af u_x+ \as u_y+bu=f, \quad
\varepsilon^{-1}p=  u_x,            \quad
\varepsilon^{-1}q=  u_y,            \quad  \textrm{in} \ \Omega.
\end{align}
Let $\dual{ \cdot}{\cdot}_{D}$ be the inner product in $L^2(D)$.
Then, the semi-discrete LDG method reads:\\
For any $t\in(0,T]$, find
such that in each element $K_{ij}$ the following variational equations
\begin{subequations}\label{LDG:scheme:2d}
\begin{alignat}{1}
 &\dual{\uN_{t}}{\vphu}_{K_{ij}}
 +\dual{(b-\af_x-\as_y)\uN}{\vphu}_{K_{ij}}
 \nonumber\\
 &-\dual{\af \uN-\pN}{\vphu_x}_{K_{ij}}
 +\dual{\af_{i,y}\widetilde{\uN}_{i,y}-\widehat{\pN}_{i,y}}{\vphu^{-}_{i,y}}_{J_j}
 -\dual{\af_{i-1,y}\widetilde{\uN}_{i-1,y}-\widehat{\pN}_{i-1,y}}{\vphu^{+}_{i-1,y}}_{J_j}
 \nonumber\\
 &-\dual{\as \uN-\qN}{\vphu_y}_{K_{ij}}
 +\dual{\as_{x,j}\widetilde{\uN}_{x,j}-\widehat{\qN}_{x,j}}{\vphu^{-}_{x,j}}_{I_i}
 -\dual{\as_{x,j-1}\widetilde{\uN}_{x,j-1}-\widehat{\qN}_{x,j-1}}{\vphu^{+}_{x,j-1}}_{I_i}
 \nonumber\\
 &\hspace{10.3cm} =\dual{f}{\vphu}_{K_{ij}},
 \\
  &\varepsilon^{-1}\dual{\pN}{\vphp}_{K_{ij}}+
  \dual{ \uN}{\vphp_x}_{K_{ij}}
  -\dual{\widehat{\uN}_{i,y}}{\vphp^{-}_{i,y}}_{J_j}
  +\dual{\widehat{\uN}_{i-1,y}}{\vphp^{+}_{i-1,y}}_{J_j}
  =0,
\\
&\varepsilon^{-1}\dual{\qN}{\vphq}_{K_{ij}}+
  \dual{ \uN}{\vphq_y}_{K_{ij}}
  -\dual{\widehat{\uN}_{x,j}}{\vphq^{-}_{x,j}}_{I_i}
  +\dual{\widehat{\uN}_{x,j-1}}{\vphq^{+}_{x,j-1}}_{I_i}
  =0,
\end{alignat}
\end{subequations}
hold for any $\vph=(\vphu,\vphp,\vphq)\in \mathcal{W}_N^3$.
Here the numerical fluxes are defined by
\begin{subequations}\label{flux:diffusion}
\begin{align}
\label{flux:diffusion:q}
\widehat{\pN}_{i,y}
&\;
=\begin{cases}
\pN^{+}_{i,y},                                &i=0,1,\dots,N_x-1,\\
\pN^-_{N_x,y} - \lambda_{N_x,y}\uN^-_{N_x,y}, &i=N_x,
\end{cases}
\\
\label{flux:diffusion:u}
\widehat{\uN}_{i,y}
&\;
=\begin{cases}
0,             &\hspace{1.5cm}  i=0,N_x,\\
\uN^{-}_{i,y}, &\hspace{1.5cm}  i=1,2,\dots,N_x-1,\\
\end{cases}
\\
\label{flux:convection:u}
\widetilde{\uN}_{i,y}
&\;
=\begin{cases}
0,             &\hspace{1.5cm}  i=0,\\
\uN^{-}_{i,y}, &\hspace{1.5cm}  i=1,2,\dots,N_x,
\end{cases}
\end{align}
\end{subequations}
for $y\in J_j$ and $j=1,2,\dots,N_y$.
We use $\lambda_{N_x,y}=\varepsilon/h_{x,N_x}$.
Analogously, for $x\in I_i$ and $i=1,2,\dots,N_x$,
we can define $\widehat{\qN}_{x,j}$, $\widehat{\uN}_{x,j}$
and $\widetilde{\uN}_{x,j}$, $j=0,1,\dots,N_y$.

To complete the definition of the LDG method,
we take the initial value\\
$\uN(0)=\pi u_0$
as the local $L^2$-projection of $u_0$.
That means, $\pi u_0$ is the unique function in $\mathcal{W}_N$,
such that in each element $K_{ij}$ it holds
\begin{equation}\label{l2:projection:2d}
\dual{\pi u_0}{\vphu}_{K_{ij}}=\dual{u_0}{\vphu}_{K_{ij}},
\quad \forall \vphu\in \mathcal{Q}^k(K_{ij}).
\end{equation}

For notational convenience, we use
$\dual{w}{v}=\sum_{K_{ij}\in \Omega_N}\dual{w}{v}_{K_{ij}}$
and write the above LDG method into a compact form:\\
For any $t\in(0,T]$, find $\wN=(\uN,\pN,\qN)\in \mathcal{W}_N^3$, such that
\begin{equation}\label{compact:form:2d}
\dual{\uN_{t}}{\vphu}+B(\wN;\vph)=\dual{f}{\vphu}
\quad \forall \vph=(\vphu,\vphp,\vphq)\in \mathcal{W}_N^3,
\end{equation}
where
\begin{align}
\label{B:def:2d}
B(\wN;\vph)=&\;
\mathcal{T}_1(\wN;\vph)+\mathcal{T}_2(\uN;\vph)
+\mathcal{T}_3(\wN;\vphu)+\mathcal{T}_4(\uN;\vphu),
\end{align}
with
\begin{align}
\mathcal{T}_1(\wN;\vph)=&\;
\varepsilon^{-1}[\dual{\pN}{\vphp}+\dual{\qN}{\vphq}]+\dual{(b- \af_x-\as_y)\uN}{\vphu},
\nonumber\\
\mathcal{T}_2(\uN;\vph)=&\;
\dual{ \uN}{\vphp_x}
+\sum_{j=1}^{N_y}\sum_{i=1}^{N_x-1}\dual{\uN^{-}_{i,y}}{\jump{\vphp}_{i,y}}_{J_j}
+\dual{ \uN}{\vphq_y}
+\sum_{i=1}^{N_x}\sum_{j=1}^{N_y-1}\dual{\uN^{-}_{x,j}}{\jump{\vphq}_{x,j}}_{I_i},
\nonumber\\
\mathcal{T}_3(\wN;\vphu)=&\;
\dual{\pN}{\vphu_x}
+\sum_{j=1}^{N_y}\Big[
\sum_{i=0}^{N_x-1}\dual{\pN^{+}_{i,y}}{\jump{\vphu}_{i,y}}_{J_j}
-\dual{\pN^{-}_{N_x,y}}{\vphu^{-}_{N_x,y}}_{J_j}
\Big]
\nonumber\\
+&\dual{\qN}{\vphu_y}
+\sum_{i=1}^{N_x}\Big[
\sum_{j=0}^{N_y-1}\dual{\qN^{+}_{x,j}}{\jump{\vphu}_{x,j}}_{I_i}
-\dual{\qN^{-}_{x,N_y}}{\vphu^{-}_{x,N_y}}_{I_i}
\Big],
\nonumber\\
\mathcal{T}_4(\uN;\vphu)=&\;
-\dual{\af\uN}{\vphu_x}
-\sum_{j=1}^{N_y}\Big[\sum_{i=1}^{N_x}
\dual{\af_{i,y}\uN^{-}_{i,y}}{\jump{\vphu}_{i,y}}_{J_j}
-\dual{\lambda_{N_x,y}\uN^{-}_{N_x,y}}{\vphu^{-}_{N_x,y}}_{J_j}
\Big]
\nonumber\\
&-\dual{\as\uN}{\vphu_y}
-\sum_{i=1}^{N_x}\Big[\sum_{j=1}^{N_y}
\dual{\as_{x,j}\uN^{-}_{x,j}}{\jump{\vphu}_{x,j}}_{I_i}
-\dual{\lambda_{x,N_y}\uN^{-}_{x,N_y}}{\vphu^{-}_{x,N_y}}_{I_i}
\Big].
\nonumber
\end{align}
Induced by \eqref{B:def:2d}, we define an energy norm $\enorm{\wN}^2\equiv B(\wN;\wN)$, that means,
\begin{align}
\label{energy:norm:2d}
\enorm{\wN}^2=&\;
\varepsilon^{-1}\norm{\pN}^2+\varepsilon^{-1}\norm{\qN}^2+\norm{(b-\af_x/2-\as_y/2)^{1/2}\uN}^2
\nonumber\\
&+\sum_{j=1}^{N_y}
\Big[\sum_{i=0}^{N_x}\frac12 \dual{\af_{i,y}}{\jump{\uN}^2_{i,y}}_{J_j}
+\dual{\lambda_{N_x,y}}{\jump{\uN}^2_{N_x,y}}_{J_j}
\Big]
\nonumber\\
&+\sum_{i=1}^{N_x}
\Big[\sum_{j=0}^{N_y}\frac12 \dual{\as_{x,j}}{\jump{\uN}^2_{x,j}}_{I_i}
+\dual{\lambda_{x,N_y}}{\jump{\uN}^2_{x,N_y}}_{I_i}
\Big].
\end{align}

Next we introduce a fully-discrete LDG method,
which is the combination of the semi-discrete LDG method \eqref{compact:form:2d} and an implicit $\theta$-scheme with $\theta\in[1/2,1]$.

Let $M$ be a positive integer and $0=t^0<t^1<\dots<t^M=T$
be an equidistant partition of $[0,T]$.
Define the time interval $K^m=(t^{m-1},t^m]$, $m=1,2,\dots,M$,
with the mesh width $\Delta t=t^m-t^{m-1}$ satisfying $M\Delta t =T$.
Denote by $v^{m}=v(t^m)$ and $v^{m,\theta}=\theta v^{m}+(1-\theta)v^{m-1}$.

The fully-discrete LDG $\theta$-scheme for \eqref{cd:spp:2d} reads:\\
Take $\uph^0 = \pi u_0$ as before.
For any $m=1,2,\dots,M$, find
$\wph^{m}=(\uph^m,\pph^m,\qph^m)\in \mathcal{W}_N^3$ such that
\begin{align}\label{fully:theta:scheme}
\dual{\frac{\uph^m-\uph^{m-1}}{\Delta t}}{\vphu}+B(\wph^{m,\theta};\vph)
=\dual{f^{m,\theta}}{\vphu},
\end{align}
holds for any $\vph=(\vphu,\vphp,\vphq)\in \mathcal{W}_N^3$,
where $B(\cdot;\cdot)$ is defined in \eqref{B:def:2d}.

\section{Layer-adapted meshes}
\label{sec:meshes}

In this section, we introduce a class of layer-adapted meshes based on some precise
information on the exact solution of \eqref{cd:spp:2d} and its derivatives.


\begin{proposition}\label{thm:reg:2d}
Assume $u\in C^{l+\kappa}$ in $\Omega\times(0,T)$ with some adequate positive integer $l$
and $0<\kappa<1$.
Moreover, assume the existence of a decomposition of the solution into a smooth term $S$
and layer components
\begin{equation}
u = S+E_{21}+E_{12}+E_{22},
\end{equation}
where $S$ and the layer components satisfy
\begin{subequations}\label{reg:u:2d}
\begin{align}
\left| \partial_x^{i}\partial_y^{j}\partial_t^{m}S(x,y,t)\right|
\leq&\; C,
\\
\left| \partial_x^{i}\partial_y^{j}\partial_t^{m}E_{21}(x,y,t)\right|
\leq &\; C\varepsilon^{-i} e^{-\alpha_1(1-x)/\varepsilon},
\\
\left| \partial_x^{i}\partial_y^{j}\partial_t^{m}E_{12}(x,y,t)\right|
\leq &\; C\varepsilon^{-j} e^{-\alpha_2(1-y)/\varepsilon},
\\
\left| \partial_x^{i}\partial_y^{j}\partial_t^{m}E_{22}(x,y,t)\right|
\leq &\; C\varepsilon^{-(i+j)}e^{-[\alpha_1(1-x)+\alpha_2(1-y)]/\varepsilon},
\end{align}
\end{subequations}
 for positive integers $i,j$ and $m$ with $i+j+2m\le l$.
Here $C>0$ is a constant independent of $\varepsilon$.
\end{proposition}

Shishkin proved the existence of such a decomposition under certain conditions (smoothness of the
data, strong compatibility), see also \cite{Roos2008}, Part III, Chapter 4.

Our class of layer-adapted meshes is now constructed as follows.
For the notational simplification,
we assume that $\alpha_1=\alpha_2=\alpha$, $N_x=N_y=N$.
Let $N\geq 2$ be an even integer. We introduce the mesh points
\[
0=x_0<x_1<\cdots<x_{N-1}<x_{N}=1,\quad
0=y_0<y_1<\cdots<y_{N-1}<y_{N}=1,
\]
and consider a tensor-product mesh with mesh points $(x_i,y_j)$.
Since both meshes have the same structure we only describe the mesh in $x$-direction.

The mesh is equidistant on $[0,x_{N/2}]$ with $N/2$ elements,
but gradually divided on $[x_{N/2},1]$ with $N/2$ elements, where $x_{N/2}=1-\tau$ with
\begin{equation}\label{condition:parameters}
\tau = \min\Big\{\frac12, \frac{\sigma\varepsilon}{\alpha}\varphi\big(\frac12\big)\Big\}.
\end{equation}
Here $\sigma>0$ is a user-chosen parameter and the function $\varphi$ satisfies
\begin{align}\label{phi:assumption}
\varphi(0)=0,\quad \varphi^{\prime}>0,\quad \varphi^{\prime\prime}\geq0.
\end{align}

Assume that $\varepsilon\leq N^{-1}$ throughout the paper
and it is so small that \eqref{condition:parameters}
can be replaced by $\tau =\alpha^{-1}\sigma\varepsilon\varphi\big(1/2\big)$
as otherwise the problem can be analyzed in a classical manner.

The mesh points are given by $x_i=\lambda(i/N)$ ($i=0,1,\dots,N$) with the mesh generating function
\begin{equation}\label{layer-adapted}
\lambda(t)=
\begin{dcases}
2(1-\tau)t,
&\textrm{for}
\; t\in[0, \frac12],
\\
1-\frac{\sigma \varepsilon}{\alpha}\varphi(1-t),
& \textrm{for}\; t\in[\frac12, 1].
\end{dcases}
\end{equation}

Introduced in \cite{Roos1999}, meshes with $\varphi(1/2)=\ln N$ are called Shishkin-type meshes (S-type),
meshes with $\varphi(1/2)=\ln (1/\varepsilon)$ are Bakhvalov-type meshes (B-type). Remark that the original
Bakhvalov mesh has a continuously differentiable mesh generating function.

In the analysis of numerical methods on our class of meshes the so called mesh characterizing function $\psi$,
defined by $\psi=e^{-\varphi}$, plays an important role.

In Table \ref{table:functions}, we list three often used  layer-adapted meshes,
 Shishkin-meshes (S-mesh), Bakhvalov-Shishkin meshes (BS-mesh)
 and a Bakhvalov-type mesh (B-type mesh) together with $\psi$ and the important quantity $\max|\psi^{\prime}|$,
 which arises in error estimates.
Figure \ref{figure:functions} illustrates
the generating functions on these meshes and the generated meshes,
when $\varepsilon=10^{-2}$ and $N=16$ are chosen.
For a survey concerning layer-adapted meshes see also \cite{Linss2003book}.

\begin{table}[ht]
\caption{Three layer-adapted meshes.}
\label{table:functions}
\footnotesize
\centering
\begin{tabular}{ccccc}
\toprule
& S-mesh
& BS-mesh
& B-type mesh
\\
\midrule
$\varphi(t)$             & $2t\ln N$    & $-\ln\big[1-2(1- N^{-1})t\big]$  &$-\ln\big[1-2(1-\varepsilon)t\big]$   \\
$\psi(t)$ & $N^{-2t}$    & $1-2(1-N^{-1})t$   &$1-2(1-\varepsilon)t$ \\
$\max|\psi^{\prime}|$ & $C\ln N$       & $C$ & $C$
\\
\bottomrule
\end{tabular}
\end{table}

\begin{figure}[htp]
\centering
\caption{\small
Three mesh generating functions (left) and the meshes generated (right).}
\includegraphics[width=3.2in,height=2.6in]{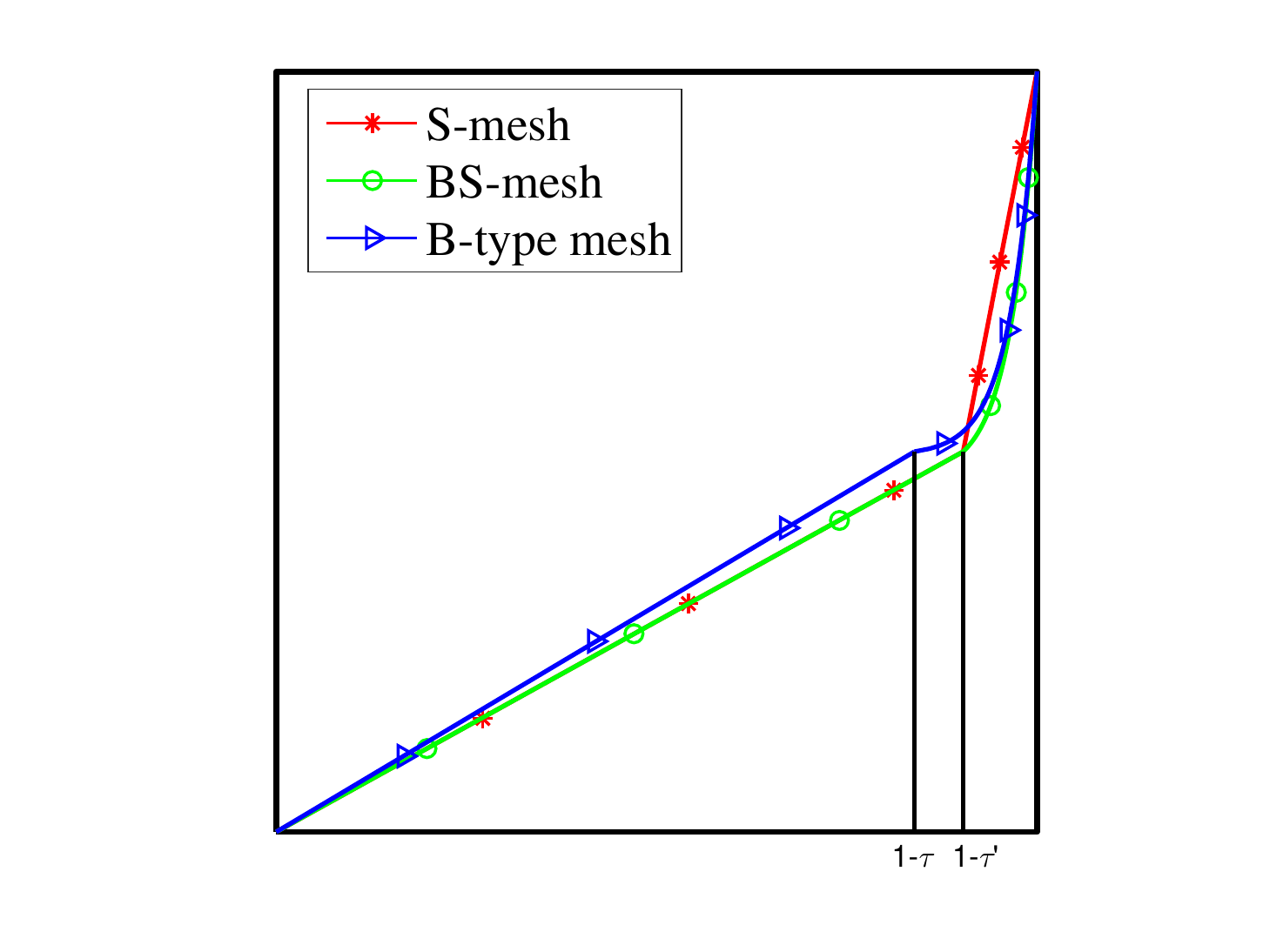}
\includegraphics[width=2.7in,height=2.6in]{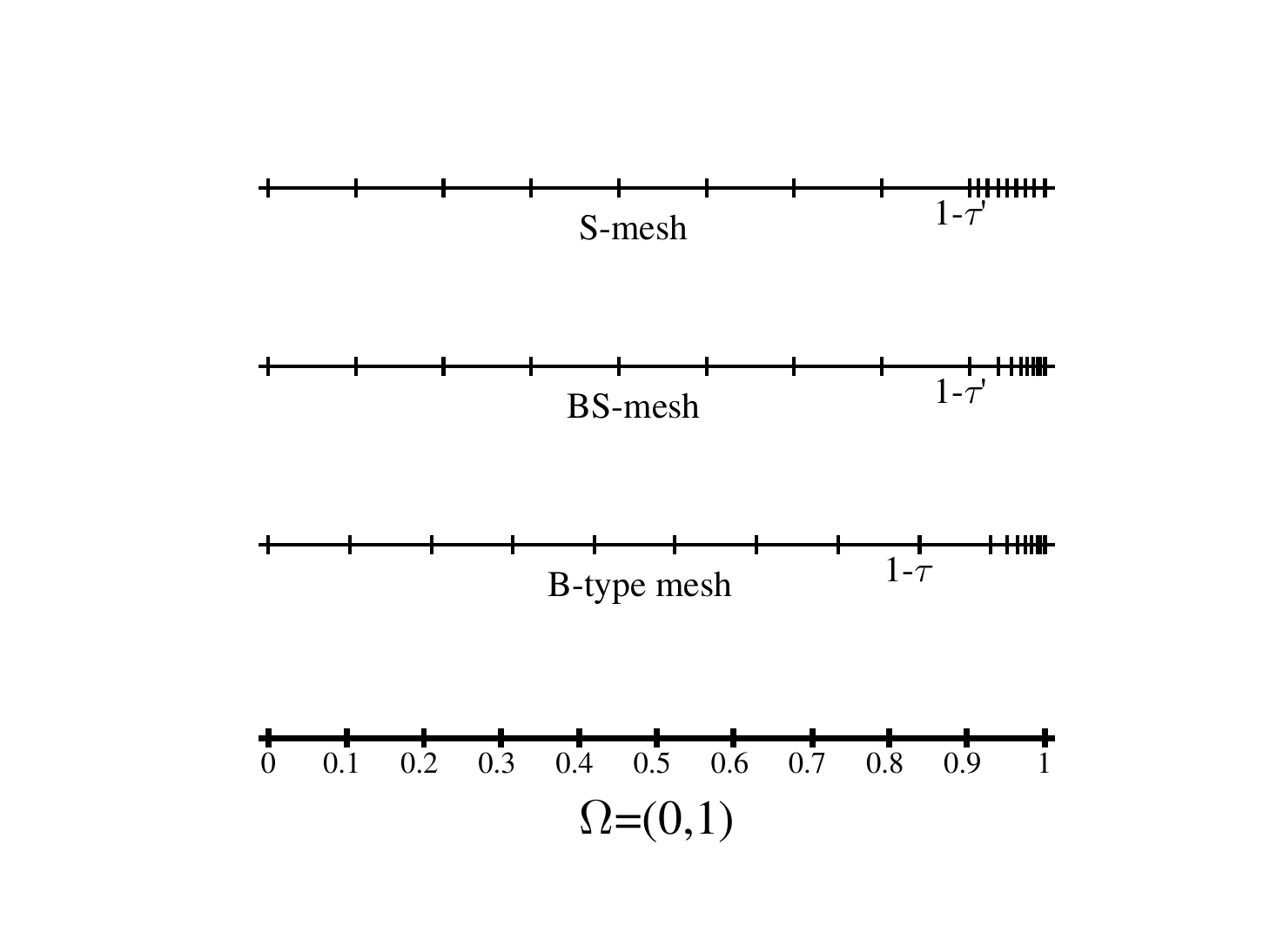}
\label{figure:functions}
\end{figure}

As mentioned already, the  two-dimensional layer-adapted mesh is obtained by the tensor product principle.
We define the following subregions:
\begin{align*}
\Omega_{11}=&\;(0,1-\tau)\times(0,1-\tau),
\quad
\Omega_{21}=(1-\tau,1)\times(0,1-\tau),
\\
\Omega_{12}=&\;(0,1-\tau)\times(1-\tau,1),
\quad
\Omega_{22}=(1-\tau,1)\times(1-\tau,1).
\end{align*}
Obviously, the elements in $\Omega_{11}$ are uniform with a mesh size
of order $N^{-1}$,
while the elements in $\Omega_{12}$ and $\Omega_{21}$
are highly anisotropic.

In the sequel, we state two preliminary lemmas,
which reflect the smallness of a weighted layer-function
and the information about the mesh size for
the  meshes considered.
Since $h_{x,i}=h_{y,i}$, $i=1,2,\dots,N$,
we simple use $h_i$ to represent one of them.

\begin{lemma}\label{lemma:1}
Denote by
$\Theta_i= \min\big\{h_i/\varepsilon,1\big\}e^{-\alpha (1-x_{i})/\sigma\varepsilon}$
for $i=N/2+1,\dots,N$.
There exists a constant $C>0$ independent of $\varepsilon$ and $N$ such that
\begin{subequations}\label{smallness:layer}
\begin{align}
\label{GJR}
\max_{N/2+1\leq i\leq N}\Theta_i \leq &\; CN^{-1}\max|\psi^{\prime}|,
\\
\label{GJR:2}
\sum_{i=N/2+1}^{N}\Theta_i       \leq &\;  C.
\end{align}
\end{subequations}
\end{lemma}

\begin{proof}
Let $t_i=i/N$. For any $i=N/2+1,\dots,N$, one has
\begin{align*}
e^{-\frac{\alpha (1-x_{i})}{\sigma\varepsilon}}
=&\; e^{-\varphi(1-t_{i})}= e^{\ln \psi(1-t_{i})} = \psi(1-t_{i}),
\\
h_i=&\;
\frac{\sigma \varepsilon}{\alpha}\big[\varphi(1-t_{i-1})-\varphi(1-t_{i})\big]
=\frac{\sigma \varepsilon}{\alpha}\int_{1-t_{i}}^{1-t_{i-1}} \varphi^{\prime}(s)ds
=\frac{\sigma \varepsilon}{\alpha}\int_{1-t_{i}}^{1-t_{i-1}}- \frac{\psi^{\prime}(s)}{\psi(s)}ds
\nonumber\\
\leq &\;
\frac{\sigma \varepsilon}{\alpha}\frac{\int_{1-t_{i}}^{1-t_{i-1}} -\psi^{\prime}(s)ds}{\psi(1-t_{i-1})}
=\frac{\sigma \varepsilon}{\alpha}\cdot\frac{\psi(1-t_{i})-\psi(1-t_{i-1})}{\psi(1-t_{i-1})},
\end{align*}
where we used the monotone decrease property
of $\psi=e^{-\varphi}$ due to $\varphi^{\prime}>0$.
Therefore, one obtains
\begin{align*}
\Theta_i
\leq&\;
C\min\Bigg\{\frac{\psi(1-t_{i})-\psi(1-t_{i-1})}{\psi(1-t_{i-1})},1\Bigg\}\psi(1-t_{i})
\nonumber\\
=&\;
C\min\Bigg\{1+\frac{\psi(1-t_{i})-\psi(1-t_{i-1})}{\psi(1-t_{i-1})},
1+\frac{\psi(1-t_{i-1})}{\psi(1-t_{i})-\psi(1-t_{i-1})}\Bigg\}
\nonumber\\
&\;\hspace{1cm}\cdot\Big[\psi(1-t_{i})-\psi(1-t_{i-1})\Big]
\nonumber\\
\leq &\; C\Big[\psi(1-t_{i})-\psi(1-t_{i-1})\Big],
\end{align*}
due to the trivial inequality $1<\min\{1+1/x,1+x\}\leq 2$ for $x>0$.
This leads immediately to \eqref{smallness:layer}.
\end{proof}

\begin{lemma}\label{lemma:2}
For our three types of layer-adapted meshes we have\\
$h_{1}=h_2=\dots=h_{N/2}$ and
$\hbar\geq C\varepsilon N^{-1}\max|\psi^{\prime}|$.
Furthermore,
\begin{align}
\label{mesh:property:S}
h_{N/2+1}=h_{N/2+2}=\dots=&\;h_{N},
&\quad& \textrm{for \; a\; S-mesh},
\\
\label{mesh:property:BS}
i= N/2+1,\cdots,N,\qquad
1\geq \frac{h_{i+1}}{h_{i}}\geq&\; C,
&\quad& \textrm{for \; a\; BS-mesh},
\\
\label{mesh:property:B:1}
i= N/2+2,\cdots,N,\qquad
1\geq \frac{h_{i+1}}{h_{i}}\geq&\; C,
&\quad& \textrm{for \; a\; B-type mesh},
\\
\label{mesh:property:B}
i= 1,2,\cdots,N/2,\qquad
h_{N/2+i}\geq&\; \frac{\sigma\varepsilon}{\alpha}\frac{1}{i+1},
&\quad& \textrm{for \; a\; B-type mesh},
\end{align}
where $C>0$ is a constant independent of $\varepsilon$ and $N$.
\end{lemma}

\begin{proof}
It is obvious that $h_{1}=h_2=\dots=h_{N/2}$.
Moreover, \eqref{mesh:property:S}-\eqref{mesh:property:B:1} can be verified easily, see also \cite{Roos2012}.

By the definition \eqref{layer-adapted}, assumption \eqref{phi:assumption}
and $\psi=e^{-\varphi}$, one has
\begin{align}
\hbar \geq C\varepsilon N^{-1}\min|\varphi^{\prime}|
=C\varepsilon N^{-1}|\varphi^{\prime}(0)|
=C\varepsilon N^{-1}|\psi^{\prime}(0)|
=C\varepsilon N^{-1}\max|\psi^{\prime}|.
\end{align}

Furthermore, on the B-type mesh, it holds
\begin{align*}
h_{N/2+i}=&\;
\frac{\sigma\varepsilon}{\alpha}
\ln\Bigg(\frac{1-2(1-\varepsilon)(1-t_{N/2+i})}{1-2(1-\varepsilon)(1-t_{N/2+i-1})}\Bigg)
=\frac{\sigma\varepsilon}{\alpha}
\ln\Bigg(1+\frac{2(1-\varepsilon)}{N\varepsilon+2(1-\varepsilon)(i-1)}\Bigg)
\\
\geq&\; \frac{\sigma\varepsilon}{\alpha}\ln\Big(1+\frac{1}{i}\Big)
\geq \frac{\sigma\varepsilon}{\alpha}\frac{1}{i+1},
\quad i= 1,2,\cdots,N/2,
\end{align*}
because $0<N\varepsilon\leq 1\leq 2(1-\varepsilon)$
and $\ln(1+x)\geq x/(1+x)$ for any $x>0$.
\end{proof}

\section{A Ritz projection}
\label{sec:projections}

In this section we introduce a Ritz projection
related to the stationary problem
and establish its approximation properties.

For this purpose, we define three two-dimensional local Gauss-Radau projections. To be specific,
for each element $K_{ij}\in \Omega_N$ and any $z\in C(\overline{K}_{ij})$,
projections $\Pi^{-}z,\Pi_x^{+} z,\Pi_y^{+} z\in \mathcal{Q}^k(K_{ij})$
are defined as
\begin{align}\label{global:projection:2d:u}
&
\begin{cases}
\int_{K_{ij}}(\Pi^{-}z) \vphu \textrm{d}x\textrm{d}y
= \int_{K_{ij}}z \vphu \textrm{d}x\textrm{d}y,
\quad &\forall \vphu\in \mathcal{Q}^{k-1}(K_{ij}),
\\
\int_{J_j}(\Pi^{-} z)_{i,y}^{-}\vphu\textrm{d}y
=\int_{J_j} z_{i,y}^{-}\vphu\textrm{d}y,
\quad &\forall \vphu\in \mathcal{P}^{k-1}(J_j),
\\
\int_{I_i}(\Pi^{-} z)_{x,j}^{-}\vphu\textrm{d}x
=\int_{I_i} z_{x,j}^{-}\vphu\textrm{d}x,
\quad &\forall \vphu\in \mathcal{P}^{k-1}(I_i),
\\
(\Pi^{-} z)(x_{i}^{-},y_{j}^{-})=z(x_{i}^{-},y_{j}^{-}).
\end{cases}
\\
&
\label{global:projection:2d:q1}
\begin{cases}
\int_{K_{ij}}(\Pi_x^{+} z) \vphu
\textrm{d}x\textrm{d}y
=
\int_{K_{ij}}z \vphu \textrm{d}x\textrm{d}y,
\quad &\forall \vphu\in \mathcal{P}^{k-1}(I_i)\otimes \mathcal{P}^k(J_j),
\\
\int_{J_j}
(\Pi_x^{+} z)_{i,y}^{+}\vphu
\textrm{d}y
=\int_{J_j}  z_{i,y}^{+}\vphu\textrm{d}y,
\quad &\forall \vphu\in \mathcal{P}^k(J_j).
\end{cases}
\\
\label{global:projection:2d:q2}
&
\begin{cases}
\int_{K_{ij}}(\Pi_y^{+}  z) \vphu
\textrm{d}x\textrm{d}y
=
\int_{K_{ij}}z \vphu \textrm{d}x\textrm{d}y,
\quad &\forall \vphu\in \mathcal{P}^{k}(I_i)\otimes \mathcal{P}^{k-1}(J_j),
\\
\int_{I_i}
(\Pi_y^{+} z)^{+}_{x,j}\vphu
\textrm{d}x
=\int_{I_i}  z^{+}_{x,j}\vphu\textrm{d}x,
\quad &\forall \vphu\in \mathcal{P}^k(I_i).
\end{cases}
\end{align}
Using these definitions and Lemma 2.14 of \cite{Apel1999}
we fix the $L^\infty$-stability property
and the approximation error estimates
\begin{subequations}\label{2GR:stb:app}
\begin{align}
\label{2GR:stb}
\norm{\Phi z}_{L^\infty(K_{ij})}\leq &\; C \norm{z}_{L^\infty(K_{ij})},
\\
\label{2GR:app}
\norm{z-\Phi z}_{L^\infty(K_{ij})}\leq &\;
C \big[h_{i}^{k+1}\norm{\partial_x^{k+1}z}_{L^\infty(K_{ij})}
+h_{j}^{k+1}\norm{\partial_y^{k+1}z}_{L^\infty(K_{ij})}
\big],
\end{align}
\end{subequations}
where $\Phi\in\{\Pi^{-},\Pi_x^{+},\Pi_y^{+}\}$.

The next lemma states some  approximation error estimates for
$\zeta_u=u-\Pi^{-}u$, $\zeta_p=p-\Pi_x^{+}p$ and $\zeta_q=q-\Pi_y^{+}q$.

\begin{lemma}\label{lemma:GR}
Let $\sigma\geq k+1$. Then it holds
\begin{subequations}\label{2dGR:property}
\begin{align}
\label{etau:2d}
\norm{\zeta_u}_{\Omega_{11}}\leq CN^{-(k+1)},
\quad
\norm{\zeta_u}_{L^\infty(\Omega\setminus\Omega_{11})}\leq &\;C(N^{-1}\max|\psi^{\prime}|)^{k+1},
\\
\label{etau:2d:1}
\sum_{j=1}^{N}\norm{(\zeta_u)_{N,y}^{-}}^2_{J_j}
+\sum_{i=1}^{N}\norm{(\zeta_u)_{x,N}^{-}}^2_{I_i}
\leq &\; C(N^{-1}\max|\psi^{\prime}|)^{2(k+1)},
\\
\label{etau:2d:2}
\sum_{i=1}^{N}\sum_{j=1}^{N}
\Big[
\norm{(\zeta_u)_{i,y}^{-}}_{J_j}^2
+\norm{(\zeta_u)_{x,j}^{-}}_{I_i}^2\Big]
\leq &\; C(N^{-1}\max|\psi^{\prime}|)^{2k+1},
\\
\label{etau:2d:3}
\sum_{j=1}^{N}\sum_{i=0}^{N}
\dual{1}{\jump{\zeta_u}^2_{i,y}}_{J_j}
+\sum_{i=1}^{N}\sum_{j=0}^{N}
\dual{1}{\jump{\zeta_u}^2_{x,j}}_{I_i}
\leq &\; C(N^{-1}\max|\psi^{\prime}|)^{2k+1},
\\
\label{etap:2d:1}
\varepsilon^{-\frac12}\norm{\zeta_q}
+\varepsilon^{-\frac12}\norm{\zeta_p}\leq &\; C(N^{-1}\max|\psi^{\prime}|)^{k+1},
\\
\label{etap:2d:2}
\sum_{i=1}^{N}\norm{(\zeta_q)_{x,N}^{-}}^2_{I_i}
+\sum_{j=1}^{N}\norm{(\zeta_p)_{N,y}^{-}}^2_{J_j}
\leq &\; C(N^{-1}\max|\psi^{\prime}|)^{2(k+1)},
\end{align}
\end{subequations}
where 
$C>0$ is independent of $\varepsilon$ and $N$.
\end{lemma}

\begin{proof}
Let us first prove \eqref{etau:2d}
based on  the decomposition $u=S+E_{21}+E_{12}+E_{22}$.
Obviously, it holds for the smooth component $S$
due to the approximation property \eqref{2GR:app} and the estimates \eqref{reg:u:2d}.

\begin{subequations}
Next we have to study each layer component separately.
For $K_{ij}\in \Omega_{11}\cup \Omega_{12}$,
one obtains from the $L^\infty$-stability \eqref{2GR:stb}
\begin{align}\label{ineq:0}
\norm{\zeta_{E_{21}}}^2_{L^\infty(K_{ij})}
\leq C \norm{E_{21}}^2_{L^\infty(K_{ij})}
\leq C e^{-2\alpha(1-x_i)/\varepsilon}
\leq C e^{-2\alpha\tau/\varepsilon}
\leq C N^{-2(k+1)},
\end{align}
due to $\sigma\geq k+1$ and $\varphi(1/2)\geq \ln N$.
For $K_{ij}\in \Omega_{21}\cup \Omega_{22}$,
one obtains from the stability  \eqref{2GR:stb}
and the approximation property \eqref{2GR:app}
\begin{align}\label{2dlayer:estimate}
&\norm{\zeta_{E_{21}}}^2_{L^\infty(K_{ij})}
\nonumber\\
&\leq
C\min\Big\{
\norm{E_{21}}^2_{L^\infty(K_{ij})},
h_{i}^{2(k+1)}\norm{\partial_x^{k+1}E_{21}}^2_{L^\infty(K_{ij})}
+h_{j}^{2(k+1)}\norm{\partial_y^{k+1}E_{21}}^2_{L^\infty(K_{ij})}
\Big\}
\nonumber\\
&\leq
C\min\Big\{1,\Big(\frac{h_{i}}{\varepsilon}\Big)^{2(k+1)}+h_{j}^{2(k+1)}\Big\}
e^{-2\alpha(1-x_i)/\varepsilon}
\nonumber\\
&\leq C \Theta_i^{2(k+1)}
\leq  C (N^{-1}\max|\psi^{\prime}|)^{2(k+1)},
\end{align}
where we used \eqref{reg:u:2d}, $h_{i}/\varepsilon\geq CN^{-1}\geq Ch_{j}$,
$\sigma\geq k+1$ and Lemma \ref{lemma:1}.
Similarly, we estimate  $\zeta_{E_{12}}$.

For $K_{ij}\in \Omega_{11}\cup\Omega_{12}\cup \Omega_{21}$,
one also has  $\norm{\zeta_{E_{22}}}^2_{L^\infty(K_{ij})}\leq C N^{-2(k+1)}$.
For $K_{ij}\in \Omega_{22}$, one obtains
\begin{align}
\norm{\zeta_{E_{22}}}^2_{L^\infty(K_{ij})}
&\leq
C\min\Big\{1,\Big(\frac{h_{i}}{\varepsilon}\Big)^{2(k+1)}
+\Big(\frac{h_{j}}{\varepsilon}\Big)^{2(k+1)}\Big\}
e^{-2[\alpha(1-x_i)+\alpha(1-y_j)]/\varepsilon}
\nonumber\\
&\leq
C\Big[ \Theta_i^{2(k+1)}e^{-2\alpha(1-y_j)/\varepsilon}
+e^{-2\alpha(1-x_i)/\varepsilon}\Theta_j^{2(k+1)}\Big]
\nonumber\\
& \leq C (N^{-1}\max|\psi^{\prime}|)^{2(k+1)},
\end{align}
\end{subequations}
where we used a trivial inequality
$\min\{1,a+b\}\leq \min\{1,a\}+\min\{1,b\}$
for any real number $a,b>0$ and Lemma \ref{lemma:1}.
Summarizing, one gets \eqref{etau:2d}.

Next we prove \eqref{etau:2d:1} and \eqref{etau:2d:2}.
Noticing that $\Pi^{-}=\pi_x^{-}\otimes \pi_y^{-}$,
where $\pi_x^{-}$ is one dimensional Gauss-Radau projection
in the $x$-direction (see \eqref{GR:projection}).
Thus, for any $i=1,2,\dots,N$,
one has $(\zeta_u)_{i,y}^{-}=u_{i,y}-\pi_y^{-}(u_{i,y})$,
see \cite{Cockburn2001}.
According to the decomposition $u_{i,y} = S(x_i,y,t) +E(x_i,y,t)$
and the estimates $\left| \partial_y^{j}\partial_t^{m}S(x_i,y,t)\right|\leq C$
and
$\left| \partial_y^{j}\partial_t^{m} E(x_i,y,t)
\right|\leq C\varepsilon^{-j} e^{-\alpha(1-y)/\varepsilon}$,
one easily derives
\[
\sum_{j=1}^{N}\norm{(\zeta_u)_{i,y}^{-}}_{J_j}^2
\leq C\big[N^{-2(k+1)}+N^{-1}(N^{-1}\max|\psi^{\prime}|)^{2k+1}\big]
\]
by using the solution decomposition,
the $L^\infty$-stability and the $L^\infty$-approximation property
of one-dimensional Gauss-Radau projections.
Similarly, one can bound $\sum_{i=1}^{N}\norm{(\zeta_u)_{x,j}^{-}}_{I_i}^2$.
Then \eqref{etau:2d:1} and \eqref{etau:2d:2} follow immediately.

To prove \eqref{etau:2d:3}, one starts from the following inequality:
\begin{align*}
\sum_{j=1}^{N}\sum_{i=0}^{N}
\dual{1}{\jump{\zeta_u}^2_{i,y}}_{J_j}
\leq&\;
2\Big[\sum_{j=1}^{N}\sum_{i=1}^{N}\int_{J_j}[(\zeta_u)^+_{i-1,y}]^2 dy
+\sum_{j=1}^{N}\sum_{i=1}^{N}\int_{J_j}[(\zeta_u)^-_{i,y}]^2 dy
\Big].
\end{align*}
For the first term one notices
\[
\sum_{j=1}^{N}\sum_{i=1}^{N}\int_{J_j}[(\zeta_u)^+_{i-1,y}]^2 dy
\leq
\sum_{j=1}^{N}\sum_{i=1}^{N}h_{j}\norm{\zeta_u}^2_{L^\infty(K_{ij})}
\]
and proceeds as  in \eqref{etau:2d} using \eqref{GJR:2}.
For the second term one uses \eqref{etau:2d:2}.
Similarly $\sum_{i=1}^{N}\sum_{j=0}^{N}\dual{1}{\jump{\zeta_u}^2_{x,j}}_{I_i}$
can be bounded.
Then \eqref{etau:2d:3} follows.

The remaining inequalities of \eqref{2dGR:property}
can be proved analogously, we omit the
 details.

\end{proof}

For any $\bm w=(u,p,q)\in [\mathcal{H}^1(\Omega_N)]^3$, the Ritz projection $\digamma \bm w$ with\\
$ \digamma \bm w=(\digamma_1 u,\digamma_2 p, \digamma_3 q)\in \mathcal{W}_N^3$ is defined as
\begin{equation}\label{Def:Ritz:projection:2d}
B(\digamma \bm w;\vph)= B( \bm w;\vph)
\quad \forall \vph=(\vphu,\vphp,\vphq)\in \mathcal{W}_N^3,
\end{equation}
where $B(\cdot;\cdot)$ is given by \eqref{B:def:2d}.
 The unique existence of $\digamma \bm w$ is easy to verify.
Moreover, based on Lemma \ref{lemma:GR},
we obtain the following result:

\begin{theorem}\label{lemma:ritz:2d}
Let $\sigma\geq k+2$.
There exists a constant $C>0$ independent of $\varepsilon$ and $N$ such that
\begin{align}\label{app:Ritz:projection:2d}
\norm{u_t-\digamma_1 u_t}+\enorm{\bm w-\digamma \bm w}\leq &\; C Q^{\star}N^{-(k+1/2)},
\end{align}
where
\begin{align}\label{bounding:Q}
Q^{\star}
=\begin{cases}
(\ln N)^{k+1}                                  & \textrm{for \; a\; S-mesh},\\
1                                                 & \textrm{for \; a\; BS-mesh},\\
\max\Big\{\sqrt{N^{-1}\ln(1/\varepsilon)},1\Big\} & \textrm{for \; a\; B-type mesh}.
\end{cases}
\end{align}
\end{theorem}

\begin{proof}
We start from $\bm w-\digamma \bm w =\bm \zeta-\bm \delta$ with
\begin{subequations}\label{error:decomposition}
\begin{align}
\bm \zeta &\; = (\zeta_u,\zeta_p,\zeta_q)
= (u-\Pi^{-} u,p-\Pi_x^{+} p,q-\Pi_y^{+} q),
\\
\bm \delta &\; = (\delta_u,\delta_p,\delta_q)
= (\digamma_1 u-\Pi^{-} u,\digamma_2 p-\Pi_x^{+} p,\digamma_3 q-\Pi_y^{+} q)\in \mathcal{W}_N^3.
\end{align}
\end{subequations}
By \eqref{Def:Ritz:projection:2d}, one has
\begin{align}\label{error:equation}
B(\bm \delta;\vph)=B(\bm \zeta;\vph)
\quad \forall \vph=(\vphu,\vphp,\vphq)\in \mathcal{W}_N^3.
\end{align}
Based on \eqref{error:equation},
we  prove \eqref{app:Ritz:projection:2d} in two steps.

(1) We shall show that
\begin{align}\label{boundedness:B}
B(\bm\zeta;\vph)\leq CQ^{\star} N^{-(k+1/2)}\interleave\vph\interleave_{\sharp}
\end{align}
for any test function $\vph\in \mathcal{W}_N^3$, where
\begin{align}\label{norm:sharp}
\interleave\vph\interleave_{\sharp}\equiv&\;
C\Big[\enorm{\vph}^2
+N^{-1}\norm{\bm a\cdot\nabla \vphu}_{\Omega_{11}}^2
+\sum_{K_{ij}\in\Omega\setminus\Omega_{11}}
\tilde{h}_{ij}\norm{\nabla \vphu}_{K_{ij}}^2
\Big]^{\frac12}.
\end{align}
Here $\tilde{h}_{ij}=\min\{\hbar_{ij},\hbar_{i+1,j},\hbar_{i,j+1}\}$
and $\tilde{h}_{G}=\min_{K_{ij}\in G}\tilde{h}_{ij}$.

To get \eqref{boundedness:B}, we  bound each term in
\[
B(\bm \zeta;\vph)=\mathcal{T}_1(\bm \zeta;\vph)+\mathcal{T}_2(\zeta_u;\vph)
+\mathcal{T}_3(\bm \zeta;\vphu)+\mathcal{T}_4(\zeta_u;\vphu).
\]

Using the Cauchy-Schwarz inequality, \eqref{etau:2d} and \eqref{etap:2d:1} yield
\begin{align*}
\mathcal{T}_1(\bm \zeta;\vph)
\leq &\;
C\big(\varepsilon^{-1/2}\norm{\zeta_p}+\varepsilon^{-1/2}\norm{\zeta_q}+\norm{\zeta_u}\big)\enorm{\vph}
\nonumber\\
\leq &\; C(N^{-1}\max|\psi^{\prime}|)^{k+1}
\enorm{\vph}.
\end{align*}

By the definitions \eqref{global:projection:2d:q1}-\eqref{global:projection:2d:q2},
Cauchy-Schwarz inequality and \eqref{etap:2d:2}, one obtains easily
\begin{align}
\mathcal{T}_3(\bm \zeta;\vphu)
=&\;
-\sum_{j=1}^{N}
\dual{(\zeta_p)^{-}_{N,y}}{\vphu^{-}_{N,y}}_{J_j}
-\sum_{i=1}^{N}
\dual{(\zeta_q)^{-}_{x,N}}{\vphu^{-}_{x,N}}_{I_i}
\nonumber\\
\leq &\;C(N^{-1}\max|\psi^{\prime}|)^{k+1}\enorm{\vph}.
\end{align}

We are left to bound $\mathcal{T}_2(\zeta_u;\vph)$
and $\mathcal{T}_4(\zeta_u;\vphu)$.
Define on each element $K_{ij}$ the bilinear forms as
\begin{align*}
\mathcal{D}^1_{ij}(\zeta_u,v)
=&\;\dual{\zeta_u}{v_x}_{K_{ij}}
-\dual{(\zeta_u)^{-}_{i,y}}{v^{-}_{i,y}}_{J_{j}}
+\dual{(\zeta_u)^{-}_{i-1,y}}{v^{+}_{i-1,y}}_{J_{j}},
\\
\mathcal{D}^2_{ij}(\zeta_u,v)
=&\;\dual{\zeta_u}{v_y}_{K_{ij}}
-\dual{(\zeta_u)^{-}_{x,j}}{v^{-}_{x,j}}_{I_{i}}
+\dual{(\zeta_u)^{-}_{x,j-1}}{v^{+}_{x,j-1}}_{I_{i}}.
\end{align*}
Take $\mathcal{D}^1_{ij}(\zeta_u,v)$ as an example.
On one hand, there holds the superconvergence property
\begin{subequations}\label{2d:sup}
\begin{align}
\label{2d:sup:1}
|\mathcal{D}^1_{ij}(\zeta_u,v)|\leq &\;
C\sqrt{\frac{h_{j}}{h_{i}}}
\Big[h_{i}^{k+2}
\norm{\partial_x^{k+2}u}_{L^\infty(K_{ij})}
+h_{j}^{k+2}\norm{\partial_y^{k+2}u}_{L^\infty(K_{ij})}
\Big]\norm{v}_{K_{ij}},
\end{align}
for any $v\in \mathcal{Q}^k(K_{ij})$,
which follows directly from Lemma 4.8 of \cite{Zhu2014}.
On the other hand, it holds
\begin{align}
\label{2d:sup:2}
|\mathcal{D}^1_{ij}(\zeta_u,v)|\leq &\;
C\sqrt{\frac{h_{j}}{h_{i}}}\norm{u}_{L^\infty(K_{ij})}\norm{v}_{K_{ij}},
\end{align}
\end{subequations}
which is derived from the Cauchy-Schwarz inequality and inverse inequalities.

Since $u=0$ on $\partial\Omega$, 
one gets
\begin{align*}
\mathcal{T}_2(\zeta_u;\vph)
&\;=\sum_{K_{ij}\in \Omega_N}\mathcal{D}^1_{ij}(\zeta_u,\vphp)
+\sum_{K_{ij}\in \Omega_N}\mathcal{D}^2_{ij}(\zeta_u,\vphq).
\end{align*}

Now we focus on estimating
$\sum_{K_{ij}\in \Omega_N}\mathcal{D}^1_{ij}(\zeta_u,\delta_p)$
based on the decomposition\\
$u =\sum_{\varphi\in\{ S,E_{21},E_{12},E_{22}\}} \varphi$.
For $\varphi =S$,
using $h_{i}\geq C\varepsilon N^{-1}\geq C\varepsilon h_{j}$,
one gets from \eqref{2d:sup:1}
\begin{align*}
&\sum_{K_{ij}\in \Omega_N}\mathcal{D}^1_{ij}(\zeta_\varphi,\vphp)
\\
\leq&\; C\sum_{K_{ij}\in \Omega_N}\sqrt{\frac{h_{j}}{h_{i}}}
\Big[h_{i}^{k+2}
\norm{\partial_x^{k+2}\varphi}_{L^\infty(K_{ij})}
+h_{j}^{k+2}\norm{\partial_y^{k+2}\varphi}_{L^\infty(K_{ij})}
\Big]\norm{\vphp}_{K_{ij}}
\\
\leq&\; CN^{-(k+1)}\norm{\vphp}_{\Omega_{11}}
+C\Big[\sum_{K_{ij}\in \Omega\setminus\Omega_{11}}
h_{j}/h_{i}N^{-2(k+2)}
\Big]^{1/2}\norm{\vphp}_{\Omega\setminus\Omega_{11}}
\\
\leq&\; CN^{-(k+1)}\enorm{\vph}.
\end{align*}
For $\varphi =E_{21}$, using \eqref{2d:sup:2} and $\sigma\geq k+2$  yields
\begin{align*}
\sum_{K_{ij}\in \Omega_{11}\cup \Omega_{12}}
\mathcal{D}^1_{ij}(\zeta_\varphi,\vphp)
\leq&\; C\sum_{K_{ij}\in \Omega_{11}\cup \Omega_{12}}
\sqrt{\frac{h_{j}}{h_{i}}}\norm{\varphi}_{L^\infty(K_{ij})}\norm{\vphp}_{K_{ij}}
\nonumber\\
\leq&\; C\sum_{K_{ij}\in \Omega_{11}\cup \Omega_{12}}
\varepsilon^{-1/2}e^{-\alpha(1-x_i)/\varepsilon}\norm{\vphp}_{K_{ij}}
\leq CN^{-(k+1)}\enorm{\vph}.
\end{align*}
Additionally, we introduce the notation $\Omega_x=\Omega_{21}\cup \Omega_{22}$.

Using \eqref{2d:sup}, $\sigma\geq k+2$ and Lemma \ref{lemma:1} gives
\begin{align*}
\sum_{K_{ij}\in \Omega_x}\mathcal{D}^1_{ij}(\zeta_\varphi,\vphp)
\leq&\; C\sum_{K_{ij}\in \Omega_x}
\sqrt{\frac{h_{j}}{h_{i}}}
\min\Big\{
h_{i}^{k+2}\norm{\partial_x^{k+2}\varphi}_{L^\infty(K_{ij})}
+h_{j}^{k+2}\norm{\partial_y^{k+2}\varphi}_{L^\infty(K_{ij})},
\nonumber\\
&\;
\norm{\varphi}_{L^\infty(K_{ij})}
\Big\}\norm{\vphp}_{K_{ij}}
\nonumber\\
\leq &\;
C\sum_{K_{ij}\in \Omega_x}(\varepsilon \max|\psi^{\prime}|)^{-1/2}
\min\Big\{1,\Big(\frac{h_{i}}{\varepsilon}\Big)^{k+2}\Big\}
e^{-\alpha(1-x_i)/\varepsilon}\norm{\vphp}_{K_{ij}}
\nonumber\\
\leq &\;
C\sum_{K_{ij}\in \Omega_x}
(\varepsilon \max|\psi^{\prime}|)^{-1/2}\Theta_i^{k+2}\norm{\vphp}_{K_{ij}}
\nonumber\\
\leq &\;
C(\max|\psi^{\prime}|)^{-1/2}\Big(\sum_{j=1}^N\sum_{i=N/2+1}^N\Theta_i\Big)^{1/2}
\max_{N/2+1\leq i \leq N}\Theta_i^{k+3/2}
\Big(\varepsilon^{-1/2}\norm{\vphp}\Big)
\\
\leq&\; C(N^{-1}\max|\psi^{\prime}|)^{k+1}\enorm{\vph},
\end{align*}
here we used $h_{j}/h_i\leq h_{j}/\hbar\leq C(\varepsilon \max|\psi^{\prime}|)^{-1}$
by Lemma \ref{lemma:2}.
Analogously, one can bound
$\sum_{K_{ij}\in \Omega_{N}}\mathcal{D}_{ij}(\zeta_\varphi,\vphp)$
for $\varphi=E_{12}, E_{22}$.

Another term in $\mathcal{T}_2(\zeta_u;\vph)$ can be bounded similarly.
Consequently,
\begin{align}
\mathcal{T}_2(\zeta_u;\vph)\leq C(N^{-1}\max|\psi^{\prime}|)^{k+1}\enorm{\vph}.
\end{align}

Now we turn to bound the convection term $\mathcal{T}_4(\zeta_u;\vphu)$,
which needs a different treatment. Define
\begin{align}
\Omega_{11}^{\star}
=\begin{cases}
\Omega_{11},
& \textrm{for \; \; S-type meshes},
\\
(0,x_{N/2+1})\times(0,y_{N/2+1}),
& \textrm{for \; \;a\,\, B-type mesh}.
\end{cases}
\end{align}

We introduce  by $\Omega_f= \Omega^{\star}_{11}\setminus\Omega_{11}$.
Let $|\Omega_f|$ be the measure of the domain $\Omega_f$.
Then $|\Omega_f|=0$ for the S-type meshes
and $|\Omega_f|\leq C\varepsilon\ln(1/\varepsilon)$ for a B-type mesh.
By Lemma \ref{lemma:2}, one has also $\tilde{h}_{\Omega_f}\geq C\varepsilon$ for a B-type mesh.

Using \eqref{etau:2d}
and \eqref{mesh:property:S}-\eqref{mesh:property:B:1},
we have
\begin{align*}
&\dual{\af \zeta_u}{\vphu_x}+\dual{\as \zeta_u}{\vphu_y}
=\dual{\zeta_u}{\bm a\cdot\nabla \vphu}
\\
&\leq
C\Big[N \norm{\zeta_u}^2_{\Omega_{11}}
+\sum_{K_{ij}\in\Omega_f}
\tilde{h}^{-1}_{ij}\norm{\zeta_u}^2_{K_{ij}}
+\sum_{K_{ij}\in\Omega\setminus\Omega_{11}^{\star}}
\tilde{h}^{-1}_{ij}\norm{\zeta_u}^2_{K_{ij}}
\Big]^{1/2}
\interleave\vph\interleave_{\sharp}
\\
&\leq C\Bigg\{
N^{-(2k+1)}
+\Big(\tilde{h}^{-1}_{\Omega_f}|\Omega_f|
+\sum_{K_{ij}\in\Omega\setminus\Omega_{11}^{\star}}
\frac{h_{i}h_{j}}{\min\{h_{i},h_{j}\}}
\Big)
(N^{-1}\max|\psi^{\prime}|)^{2(k+1)}\Bigg\}^{1/2}
\interleave\vph\interleave_{\sharp}
\\
&\leq CQ^{\star} N^{-(k+1/2)},
\end{align*}
where $Q^{\star}$ is given by \eqref{bounding:Q}.

As for the other terms in $\mathcal{T}_4(\zeta_u;\vphu)$, one can estimate
\begin{align}
\Bigg(\sum_{i=1}^{N}\sum_{j=1}^{N}\norm{(\zeta_u)_{i,y}^{-}}_{J_j}^2\Bigg)^{1/2}\enorm{\vph}
\leq C(N^{-1}\max|\psi^{\prime}|)^{k+1/2}\enorm{\vph}
\end{align}
due to \eqref{etau:2d:2}. Consequently,
\begin{align}
\mathcal{T}_4(\zeta_u;\vphu)\leq &\; CQ^{\star} N^{-(k+1/2)}\interleave\vph\interleave_{\sharp}.
\end{align}
Summarizing, we obtain \eqref{boundedness:B}.

(2) Now we start to prove \eqref{app:Ritz:projection:2d}.
Taking $\vph=\bm \delta$ in \eqref{error:equation}, one gets
\begin{align}\label{energy:identity}
\enorm{\bm \delta}^2= B(\bm \delta;\bm \delta)
= B(\bm \zeta;\bm\delta).
\end{align}

Slightly modifying the proof of Lemma 4.6 in \cite{Zhu2014},
we can state stability with respect to the norm
$\interleave\cdot\interleave_{\sharp}$, i.e.,
\begin{equation}\label{stability:bilinear:2d}
\begin{split}
\sup_{\vph\in \mathcal{W}_N^3}\frac{B(\bm w;\vph)}{\interleave\vph\interleave_{\sharp}}
\geq&\; C\interleave\bm w\interleave_{\sharp},
\end{split}
\end{equation}
here $C>0$ is independent of $\varepsilon$ and $N$.
Thus one derives from \eqref{error:equation} and \eqref{boundedness:B}
\begin{align}\label{estimate:delta}
\enorm{\bm\delta}\leq \interleave\bm \delta\interleave_{\sharp}
\leq C\sup_{\vph\in \mathcal{W}_N^3}
\frac{B(\bm \delta;\vph)}{\interleave\vph\interleave_{\sharp}}
= C\sup_{\vph\in \mathcal{W}_N^3}
\frac{B(\bm \zeta;\vph)}{\interleave\vph\interleave_{\sharp}}
\leq CQ^{\star} N^{-(k+1/2)}.
\end{align}

Since $\enorm{\bm\zeta}$ can be estimated from \eqref{2dGR:property},
we can bound $\enorm{\bm w-\digamma \bm w}$ from \eqref{estimate:delta} and the triangle inequality.

Using \eqref{etau:2d}, one has
\[
\norm{u-\digamma_1 u}\leq
\norm{\zeta_u}+\norm{\delta_u}
\leq C[(N^{-1}\max|\psi^{\prime}|)^{k+1}+\enorm{\bm\delta}]
\leq CQ^{\star} N^{-(k+1/2)}.
\]
If we differentiate \eqref{Def:Ritz:projection:2d} with respect to $t$,
and take into account that $u_t,p_t,q_t$ satisfy similar estimates as $u$, $p$ and $q$,
we can get estimate for $\norm{u_t-\digamma_1 u_t}$
analogously as before. This completes the proof.
\end{proof}

\section{Error estimates}
\label{sec:convergence:energy:norm}

In this section we use Theorem \ref{lemma:ritz:2d}
to derive error estimates as well
for the semi-discrete as the fully-discrete LDG method.

\begin{theorem}\label{thm:semi:discrete}
Let $\bm w=(u,p,q)=(u,\varepsilon u_x,\varepsilon u_y)$
where $u$ is the exact solution of \eqref{cd:spp:2d}
satisfying Proposition \ref{thm:reg:2d}.
Let $\wN=(\uN,\pN,\qN)\in \mathcal{W}_N^3$
be the numerical solution of semi-discrete LDG method \eqref{LDG:scheme:2d}.
Then, it holds
\begin{align}\label{error:estimate:semi}
\norm{(u-\uN)(T)}+\int_0^T \enorm{\bm w-\wN} dt
\leq &\; C(1+T)Q^{\star}N^{-(k+1/2)},
\end{align}
where $Q^{\star}$ is given by \eqref{bounding:Q}
and $C>0$ is a constant independent of $\varepsilon$ and $N$.
\end{theorem}

\begin{proof}
Denote the numerical error by $\bm e= \bm w-\wN$, and divide into the two parts $\bm\eta$ and $\bm\xi$
with
 $\bm e=\bm \eta-\bm \xi$,
where $\bm \eta=\bm w- \digamma \bm w$ and $\bm \xi=\wN- \digamma \bm w$.
Here $\digamma \bm w$ is the Ritz projection of $\bm w$
defining in \eqref{Def:Ritz:projection:2d}. The consistence of the numerical fluxes in the LDG method leads to
 the error equation
\begin{equation}
\dual{e_{u,t}}{\vphu}+B(\bm e;\vph)=0
\quad \forall \vph=(\vphu,\vphp,\vphq)\in \mathcal{W}_N^3.
\end{equation}
Taking $\vph=\bm \xi$ yields
\begin{equation}\label{energy:identity:semi}
\frac12 \frac{\textrm{d}}{\textrm{d} t}\norm{\xi_u}^2
+B(\bm \xi;\bm \xi)=\dual{\eta_{u,t}}{\xi_u},
\end{equation}
since $B(\bm \eta;\bm \xi)=0$ by the definition \eqref{Def:Ritz:projection:2d}.
Using \eqref{assumption:coef:2d}, we conclude
\[
B(\bm \xi;\bm \xi)=\enorm{\bm\xi}^2
\geq \beta^{\frac12}\norm{\xi_u}\enorm{\bm\xi}.
\]
By \eqref{energy:identity:semi} and the Cauchy-Schwarz inequality one gets
\[
\norm{\xi_u}\Big[\frac{\textrm{d}}{\textrm{d} t}\norm{\xi_u}
+\beta^{\frac12}\enorm{\bm\xi}\Big]
\leq \frac12 \frac{\textrm{d}}{\textrm{d} t}\norm{\xi_u}^2+B(\bm \xi;\bm \xi)
=\dual{\eta_{u,t}}{\xi_u}\leq \norm{\eta_{u,t}}\norm{\xi_u}.
\]
Canceling $\norm{\xi_u}$ on both sides results in
\[
\frac{\textrm{d}}{\textrm{d} t}\norm{\xi_u}+\beta^{\frac12}\enorm{\bm\xi}
\leq \norm{\eta_{u,t}}\leq CQ^{\star}N^{-(k+1/2)}
\]
due to Theorem \ref{lemma:ritz:2d}.
Taking into account the choice of the initial value we have
\begin{align}
\norm{\xi_u(T)}+\beta^{\frac12}\int_0^T \enorm{\bm\xi} dt
\leq &\; \norm{\xi_u(0)}+ \int_0^T CQ^{\star}N^{-(k+1/2)} dt
\nonumber\\
\leq &\; C(1+T)Q^{\star}N^{-(k+1/2)}.
\end{align}
The use of  Theorem \ref{lemma:ritz:2d} and the triangle inequality give
finally the estimate of Theorem \ref{error:estimate:semi}.

\end{proof}

Next we study the error of the fully discrete scheme.
\begin{theorem}\label{thm:fully:discrete}
The error of fully-discrete LDG scheme \eqref{fully:theta:scheme} satisfies
\begin{align}\label{estimate:fully:2d}
\norm{u(T)-\uph^M}+\Delta t\sum_{m=1}^M \enorm{(\bm w-\wph)^{m,\theta}}
\leq &\; C(1+T)\big((\Delta t)^\nu+Q^{\star}N^{-(k+1/2)}\big),
\end{align}
where $Q^{\star}$ is given by \eqref{bounding:Q},
$\nu=2$ for $\theta=1/2$ and $\nu=1$ for $1/2<\theta\leq 1$.
Here $C>0$ is a constant independent of $\varepsilon$ and $N$.
\end{theorem}

\begin{proof}
Define $\digamma \bm w^m = \digamma\bm w(t^m)$ and
\[
\bm \eta^m=\bm w^m- \digamma\bm w^m,\;
\bm \xi^m=\wph^m- \digamma\bm w^m \in \mathcal{W}_N^3.
\]
For the exact solution it holds
\begin{equation}
\dual{u^{m,\theta}_{t}}{\vphu}+B(\bm w^{m,\theta};\vph)
=\dual{f^{m,\theta}}{\vphu}
\quad \forall \vph=(\vphu,\vphp,\vphq)\in \mathcal{W}_N^3.
\end{equation}
Subtracting this identity from \eqref{fully:theta:scheme}
and noticing the relation \\$B(\bm \eta^{m,\theta};\vph)=B(\bm w^{m,\theta}-\digamma\bm w^{m,\theta};\vph)=0$
leads us to the  the error equation
\begin{align}\label{ineq:2}
\dual{\frac{\xi_u^m-\xi_u^{m-1}}{\Delta t}}{\vphu}
+B(\bm \xi^{m,\theta};\vph)
=\dual{L^m}{\vphu}, \quad \vph=(\vphu,\vphp,\vphq)\in \mathcal{W}_N^3,
\end{align}
where
\begin{equation}
L^m = \Big(u_t^{m,\theta}-\frac{u^m-u^{m-1}}{\Delta t}\Big)
+\frac{\eta_u^m-\eta_u^{m-1}}{\Delta t}.
\end{equation}
Using  a Taylor expansion we get
\begin{align*}
u_t^{m,\theta}-\frac{u^m-u^{m-1}}{\Delta t}
=&\; \Big(\theta-\frac12\Big)\Delta t u_{tt}^m
+\Big(\frac{1-\theta}{2}u_{ttt}(s_1)-\frac16u_{ttt}(s_2)\Big)(\Delta t)^2,
\\
\frac{\eta_u^m-\eta_u^{m-1}}{\Delta t}
=&\; u_{t}(s_3)-\digamma_1 u_{t}(s_3),
\end{align*}
here  the $s_i (i=1,2,3)$ are intermediate points belonging to $(t^{m-1},t^{m})$.
By \eqref{reg:u:2d} and Theorem \ref{lemma:ritz:2d} we obtain
\[
\norm{L^m}\leq
C\Big[(\theta-\frac12)\Delta t
+ (\Delta t)^2+Q^{\star}N^{-(k+1/2)}\Big]
\equiv \mathcal{E}_{N,\Delta t}.
\]

We use the stability result
\begin{subequations}
\begin{align}
B(\bm \xi^{m,\theta};\bm \xi^{m,\theta})=\enorm{\bm \xi^{m,\theta}}^2
\geq  \beta^{\frac12}\norm{\xi^{m,\theta}_u}\enorm{\bm\xi^{m,\theta}}.
\end{align}
The application of the Cauchy-Schwarz inequality, $\theta\geq 1/2$ and the triangle inequality imply
\begin{align}
\dual{\frac{\xi_u^m-\xi_u^{m-1}}{\Delta t}}{\xi^{m,\theta}_u}
=&\; (\Delta t)^{-1}\dual{\xi_u^m-\xi_u^{m-1}}{\theta \xi_u^m+(1-\theta)\xi_u^{m-1}}
\nonumber\\
=&\; (\Delta t)^{-1}\big[\theta\norm{\xi_u^m}^2-(2\theta-1)\dual{\xi_u^m}{\xi_u^{m-1}}-(1-\theta)\norm{\xi_u^{m-1}}^2\big]
\nonumber\\
\geq &\;(\Delta t)^{-1}\big[\theta\norm{\xi_u^m}^2-(2\theta-1)\norm{\xi_u^m}\norm{\xi_u^{m-1}}-(1-\theta)\norm{\xi_u^{m-1}}^2\big]
\nonumber\\
= &\; (\Delta t)^{-1}(\norm{\xi_u^m}-\norm{\xi_u^{m-1}})(\theta\norm{\xi_u^m}+(1-\theta)\norm{\xi_u^{m-1}})
\nonumber\\
\geq &\; (\Delta t)^{-1}(\norm{\xi_u^m}-\norm{\xi_u^{m-1}})\norm{\theta \xi_u^m +(1-\theta)\xi_u^{m-1}}
\nonumber\\
= &\; (\Delta t)^{-1}(\norm{\xi_u^m}-\norm{\xi_u^{m-1}})\norm{\xi^{m,\theta}_u}.
\end{align}
\end{subequations}
Inserting $\vph=\bm \xi$  in \eqref{ineq:2}  yields
\begin{align*}
\norm{\xi^{m,\theta}_u}
\Big[(\Delta t)^{-1}(\norm{\xi_u^m}-\norm{\xi_u^{m-1}})+\beta^{\frac12}\enorm{\bm\xi^{m,\theta}}\Big]
\leq \norm{L^m}\norm{\xi^{m,\theta}_u}
\leq \mathcal{E}_{N,\Delta t}\norm{\xi^{m,\theta}_u}.
\end{align*}
or
\[
(\Delta t)^{-1}(\norm{\xi_u^m}-\norm{\xi_u^{m-1}})+\beta^{\frac12}\enorm{\bm\xi^{m,\theta}}
\leq \mathcal{E}_{N,\Delta t}.
\]
Summing over $m=1,2,\dots,M$ leads to
\[
\norm{\xi_u^M}+\beta^{\frac12} \sum_{m=1}^M \enorm{\bm\xi^{m,\theta}} \Delta t
\leq \norm{\xi_u^0}+ \Delta t \sum_{m=1}^M \mathcal{E}_{N,\Delta t}
\leq C(1+T) \mathcal{E}_{N,\Delta t},
\]
due to the choice of the initial value.
By Theorem \ref{lemma:ritz:2d} and the triangle inequality we finally get the result stated in Theorem \ref{thm:fully:discrete}.
\end{proof}

\section{Numerical experiments}
\label{sec:experiments}

In this section,
we present some numerical experiments
for the fully-discrete LDG Crank-Nicolson scheme,
i.e., for $\theta=1/2$.  All the calculations were realized in MATLAB R2015B.
The systems of linear equations resulting from the discrete problems
were solved by LU-decomposition. All integrals were evaluated
using the 5-point Gauss-Legendre quadrature rule.

Let $T=1$ and $\Omega=(0,1)^2$. We consider the test problem
\begin{equation}\label{model}
\begin{split}
&u_t-\varepsilon \Delta u+u_x+u_y + u =f(x,y,t),
\end{split}
\end{equation}
where $f$ and the initial-boundary conditions are chosen in
such a way that the exact solution is given by
\begin{align}
u(x,y,t) = 
e^ t\sin (\pi xy)(1-e^{-(1-x)/\varepsilon})(1-e^{-(1-y)/\varepsilon}).
\end{align}

The fully-discrete LDG method is implemented on the three layer-adapted meshes
listed in Table \ref{table:functions}.
We take $\sigma=k+2$ and compute the $L^2$-error
$\norm{u(T)-\uph(T)}$ and the energy-norm error
$\Delta t\sum_{m=1}^M \enorm{(\bm w-\wph)^{m,\theta}}$, respectively.
The numerical convergence order is obtained from the formula
\begin{align*}
r_2 = \frac{ \log e_{N}-\log e_{2N} }{\log 2 }
\quad \textrm{or} \quad
r_S = \frac{ \log e_{N}-\log e_{2N} }{\log p_S },
\end{align*}
where $e_N$ is the error using $N$ elements in the $x$- and $y$-direction.
For the Shishkin mesh we use a scaling with
$p_S = 2\ln N/\ln(2N)$
to compute  the numerical convergence order
with respect to the power of $\ln N/N$.

The three numerical solutions and the absolute errors
on a S-mesh, BS-mesh and B-type mesh
are plotted in Figure \ref{figure:solutions},
when $\varepsilon=10^{-3}$, $k=1$ and $N=64$.
From this picture, we can see that the LDG method with these meshes
gives well-behaved numerical solutions.
The solution on the S-mesh produces the largest errors,
while the solutions on the BS-mesh and B-type mesh
have a comparable numerical performance.

We test also the numerical convergence rate in space.
Set $\varepsilon = 10^{-8}$
and compute the $L^2$-error and energy-error for $k=1,2$.
To obtain a balanced space-time error,
we take the time step as $\Delta t=N^{-1}$ for $k=1$
and $\Delta t=N^{-1.5}$ for $k=2$.
In Tables \ref{table:1/2:L2}-\ref{table:1/2:energy},
one observes convergence rates of order $k+1$
and $k+1/2$ for the $L^2$-norm error and the energy-norm error,
which imply that the numerical error behaves as the energy-norm error estimate in
Theorem \ref{thm:semi:discrete}.

To test the convergence rate in time,
we take $k=3$ and $N=128$ such that the temporal error is dominant.
Decrease the time step from $0.5$ to $0.0625$
and compute the $L^2$-error and the energy-error, respectively.
In the Tables \ref{table:time:L2}-\ref{table:time:energy} a second order convergence rate is clearly observed
which corresponds to the theoretical prediction of Theorem \ref{thm:fully:discrete}.

Finally we investigate the robustness of the LDG method
with respect to the small parameter $\varepsilon$.
Fix $k=1$, $N=128$ we vary  $\varepsilon$ from $10^{-4}$ to $10^{-11}$.
 Table \ref{table:eps} shows that both errors are almost constant.

\begin{figure}[htp]
\centering
\caption{\small
Top: numerical solutions. Bottom: absolute errors.
Left: S-mesh; Middle: BS-mesh; Right: B-type mesh.}
\includegraphics[width=1.9in,height=1.6in]{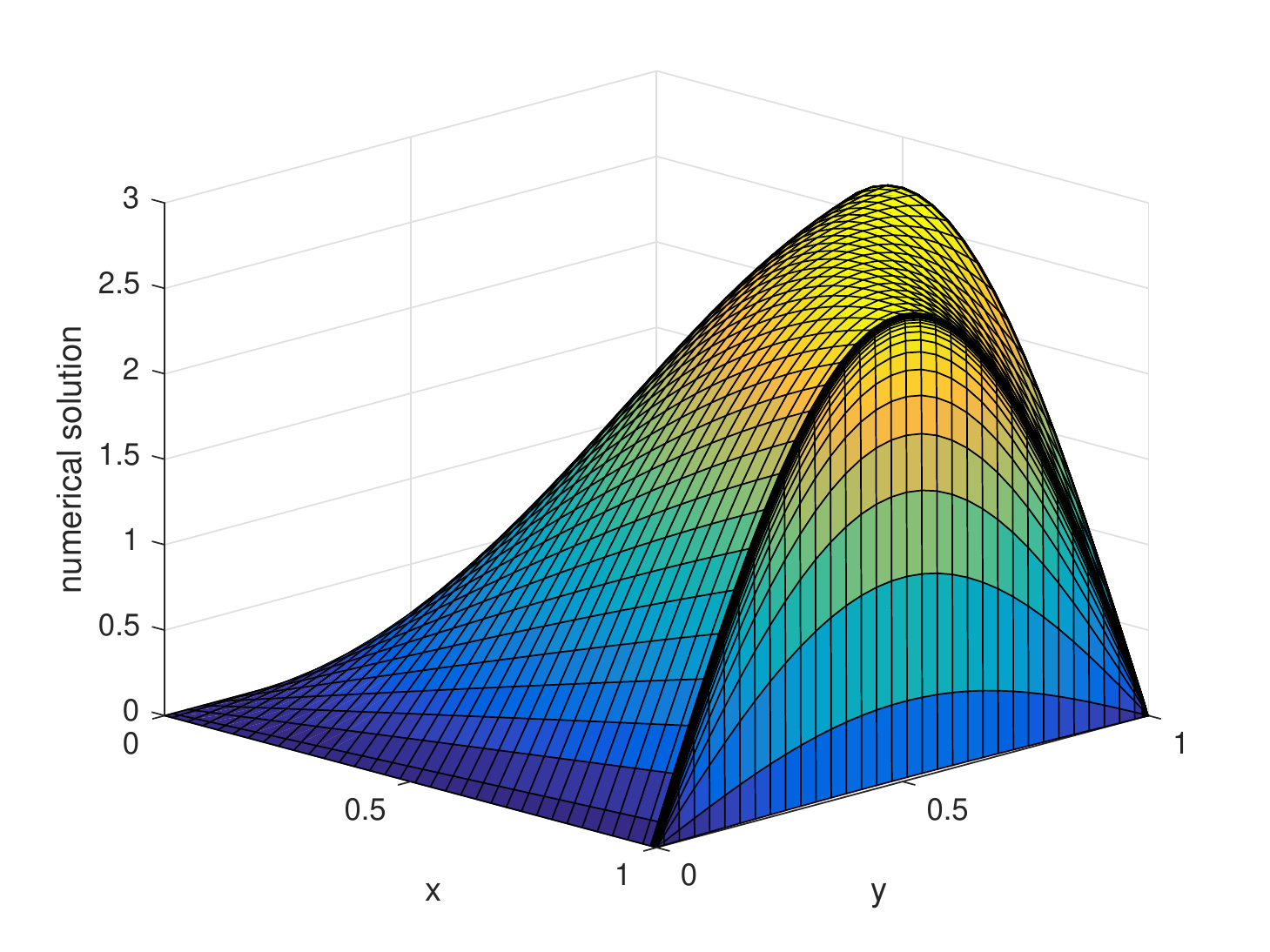}
\includegraphics[width=1.9in,height=1.6in]{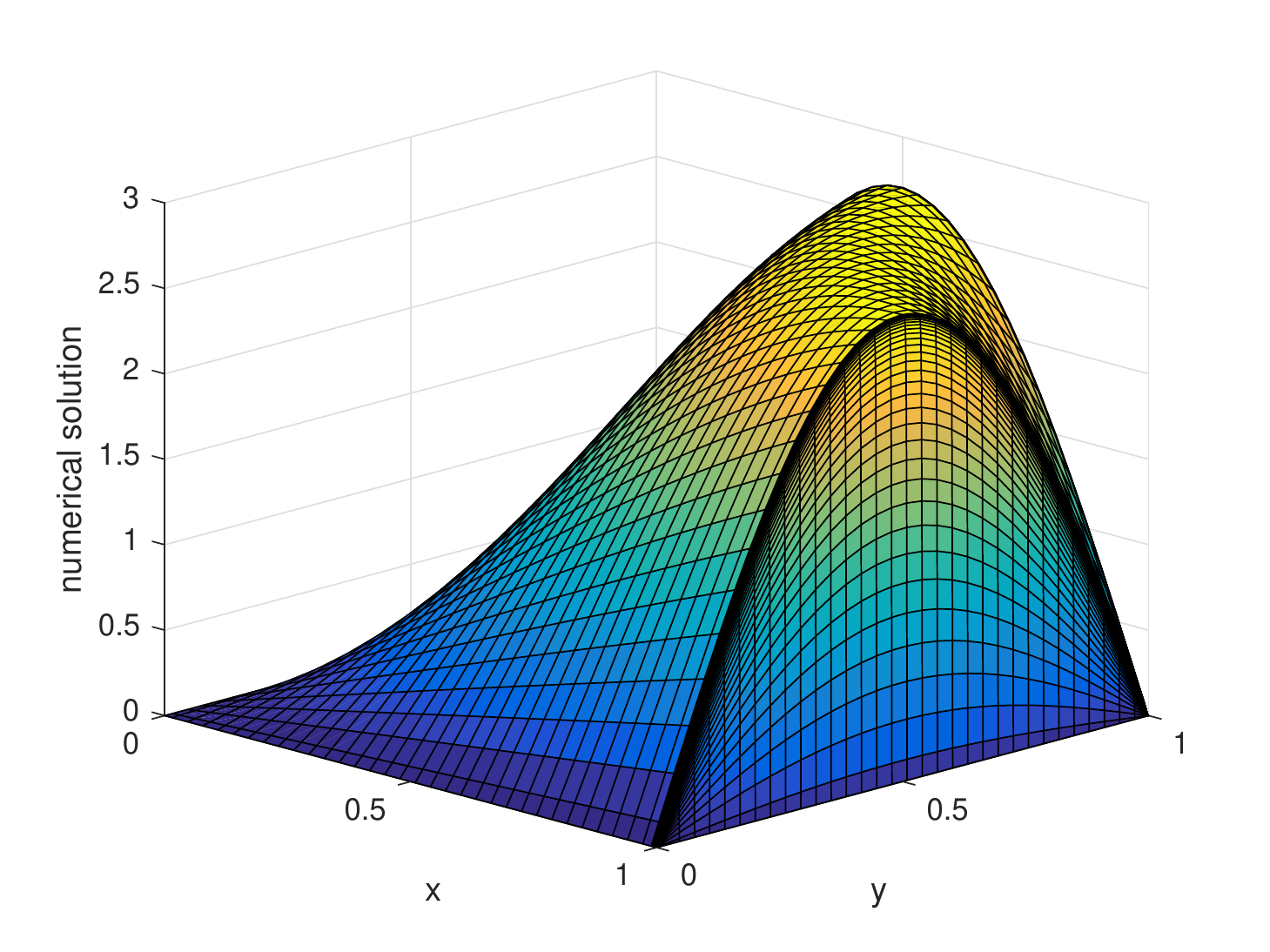}
\includegraphics[width=1.9in,height=1.6in]{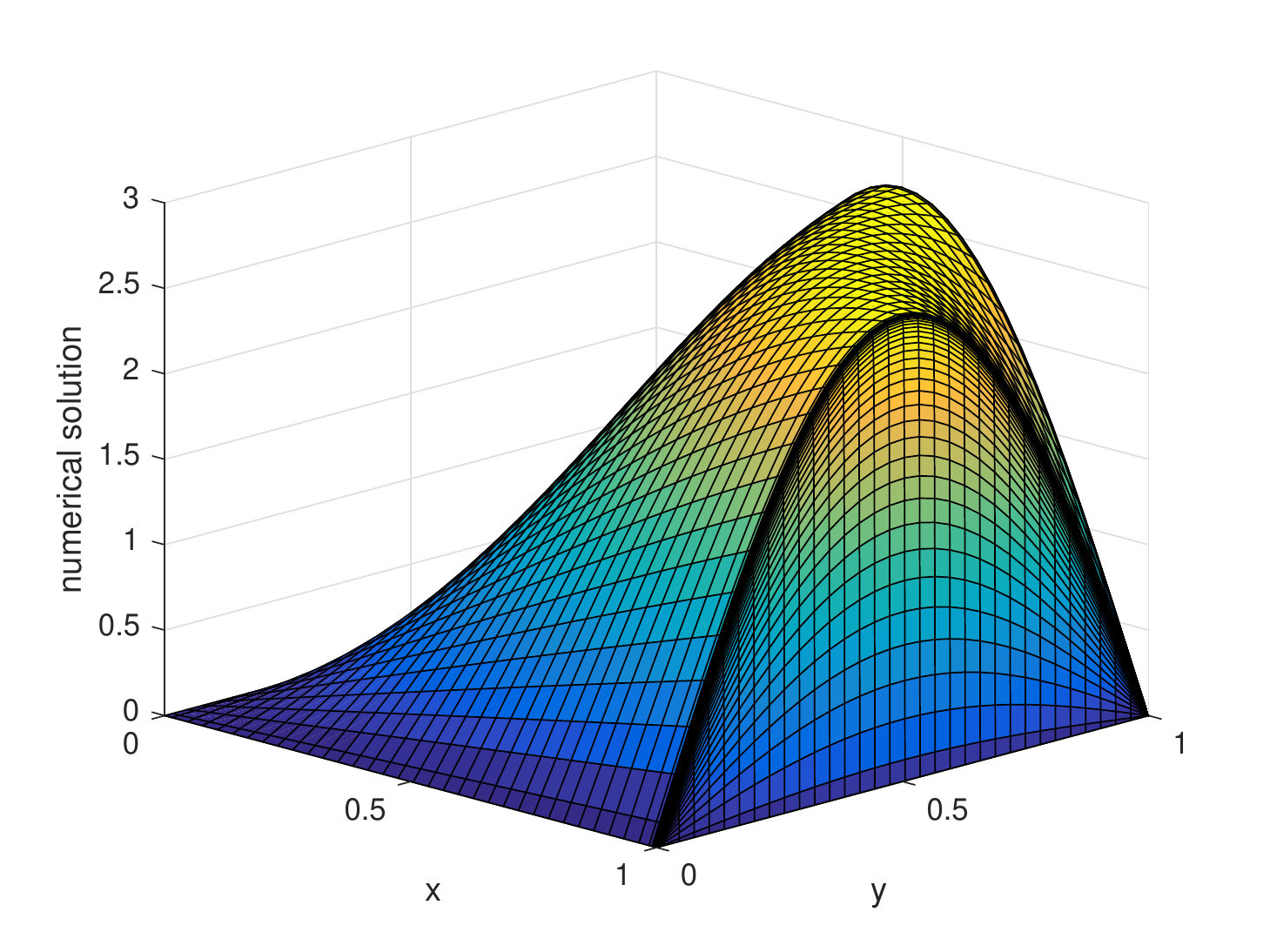}
\\
\includegraphics[width=1.9in,height=1.6in]{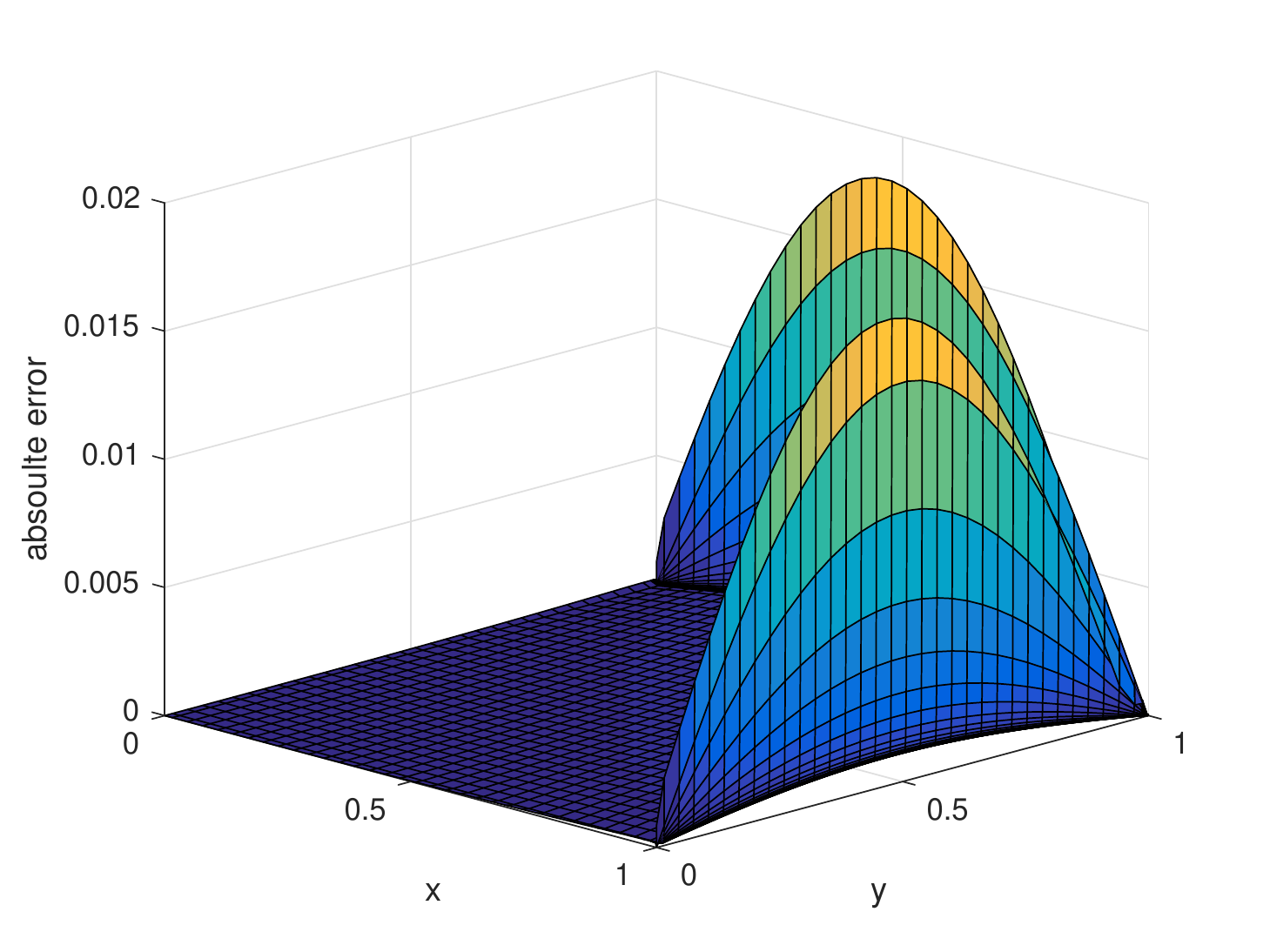}
\includegraphics[width=1.9in,height=1.6in]{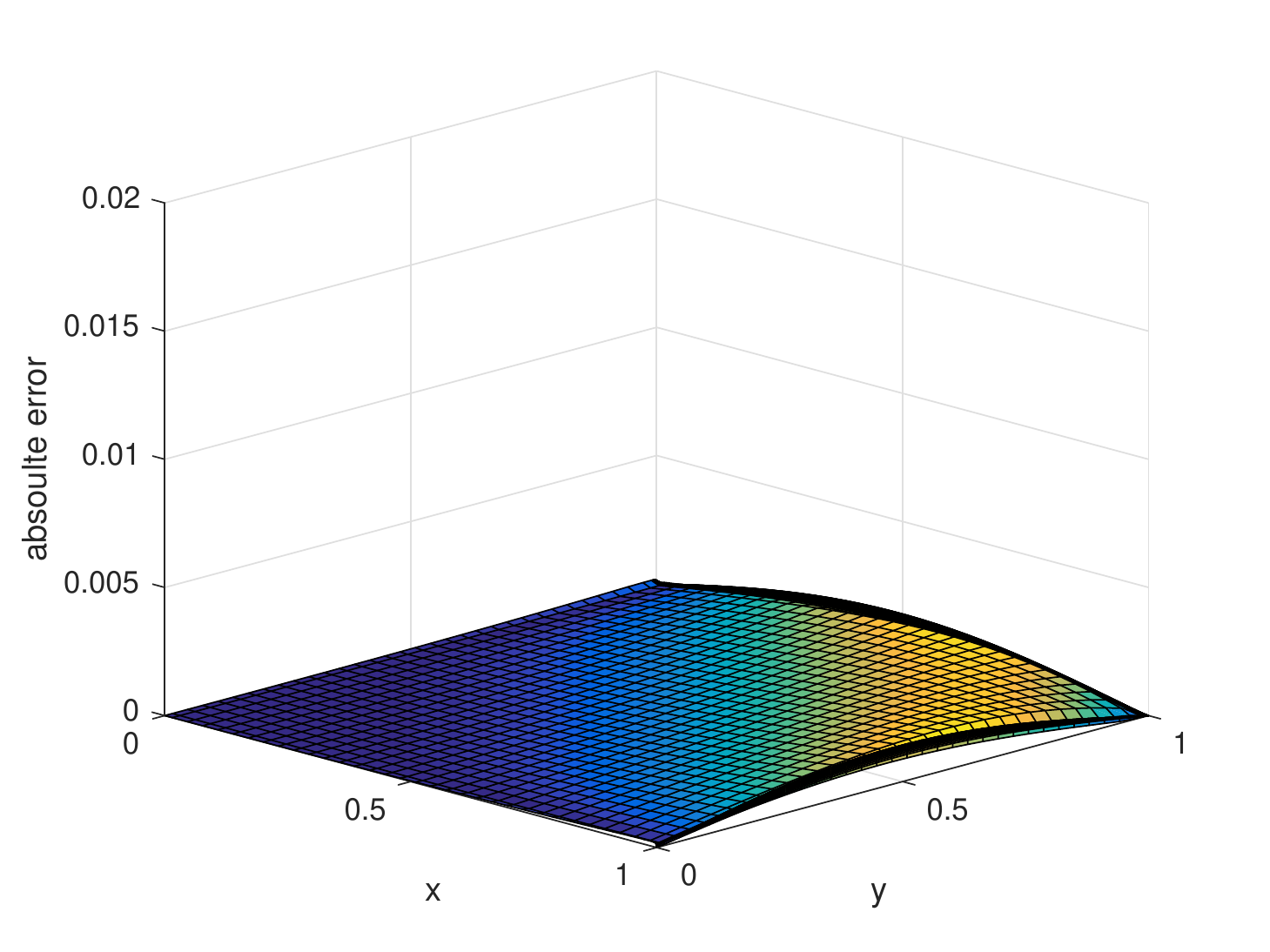}
\includegraphics[width=1.9in,height=1.6in]{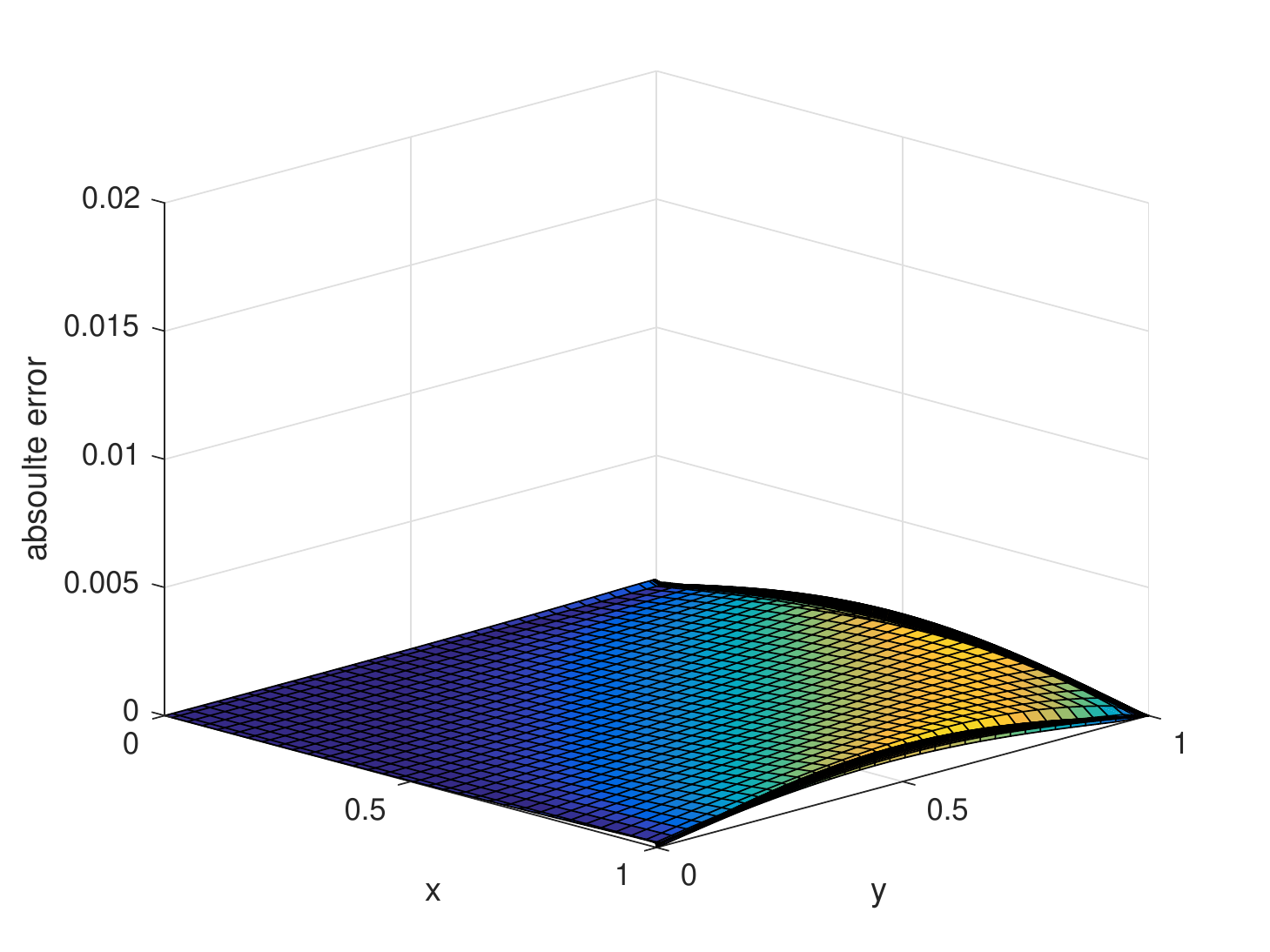}
\label{figure:solutions}
\end{figure}

\begin{table}[ht]
\normalsize
\centering
\caption{$L^2$-error and convergence rate in space.}
\label{table:1/2:L2}
\begin{tabular}{cccccccc}
\toprule
\multirow{2}*{$k$}&\multirow{2}*{$N$} & \multicolumn{2}{c}{S-mesh} & \multicolumn{2}{c}{BS-mesh} & \multicolumn{2}{c}{B-type mesh} \\
\cmidrule(r){3-4} \cmidrule(r){5-6} \cmidrule(r){7-8}
&& $L^2$-error & $r_2$ & $L^2$-error & $r_2$ & $L^2$-error & $r_2$
\\
\midrule
1
& 4 & 1.64e-01 & - & 1.62e-01 & - & 1.59e-01 & - \\
& 8 & 4.39e-02 & 1.90 & 4.35e-02 & 1.90 & 4.35e-02 & 1.87\\
& 16 & 1.14e-02 & 1.94 & 1.14e-03 & 1.93 & 1.14e-02 & 1.93\\
& 32 & 2.93e-03 & 1.97 & 2.92e-03 & 1.96 & 2.92e-03 & 1.96\\
& 64 & 7.41e-04 & 1.98 & 7.40e-04 & 1.98 & 7.40e-04 & 1.98\\
& 128 & 1.86e-04 & 1.99 & 1.86e-04 & 1.99 & 1.86e-04 & 1.99\\
2
& 4 & 1.58e-02 & - & 1.59e-02 & - & 1.55e-02 & - \\
& 8 & 2.09e-03 & 2.91 & 2.09e-03 & 2.93 & 2.10e-03 & 2.89\\
& 16 & 2.75e-04 & 2.93 & 2.74e-04 & 2.93 & 2.75e-04 & 2.93\\
& 32 & 3.52e-05 & 2.97 & 3.52e-05 & 2.96 & 3.52e-05 & 2.97\\
& 64 & 4.45e-06 & 2.98 & 4.45e-06 & 2.98 & 4.45e-06 & 2.98\\
& 128 & 5.60e-07 & 2.99 & 5.59e-07 & 2.99 & 5.59e-07 & 2.99
\\
\bottomrule
\end{tabular}
\end{table}

\begin{table}[ht]
\normalsize
\centering
\caption{Energy-error and convergence rate in space.}
\label{table:1/2:energy}
\begin{tabular}{cccccccc}
\toprule
\multirow{2}*{$k$}&\multirow{2}*{$N$} & \multicolumn{2}{c}{S-mesh} & \multicolumn{2}{c}{BS-mesh} & \multicolumn{2}{c}{B-type mesh} \\
\cmidrule(r){3-4} \cmidrule(r){5-6} \cmidrule(r){7-8}
&& energy-error& $r_S$ & energy-error & $r_2$ & energy-error & $r_2$
\\
\midrule
1
& 4  & 4.57e-01 & - & 3.77e-01 & - & 4.65e-01 & - \\
& 8  & 2.65e-01 & 1.89 & 1.52e-01 & 1.32 & 1.68e-01 & 1.47 \\
& 16 & 1.46e-01 & 1.48 & 5.76e-02 & 1.39 & 6.07e-02 & 1.47 \\
& 32 & 7.35e-02 & 1.46 & 2.12e-02 & 1.44 & 2.17e-02 & 1.48 \\
& 64 & 3.46e-02 & 1.48 & 7.64e-03 & 1.47 & 7.74e-03 & 1.49 \\
& 128 & 1.55e-02 & 1.49 & 2.73e-03 & 1.49 & 2.75e-03 & 1.49 \\
2
& 4  & 1.29e-01 & - & 7.32e-02 & - & 1.50e-01 & - \\
& 8  & 6.99e-02 & 2.13 & 1.79e-02 & 2.03 & 2.44e-02 & 2.62 \\
& 16 & 2.86e-02 & 2.21 & 3.70e-03 & 2.27 & 4.28e-03 & 2.51 \\
& 32 & 9.46e-03 & 2.35 & 7.06e-04 & 2.39 & 7.59e-04 & 2.50 \\
& 64 & 2.73e-03 & 2.44 & 1.30e-04 & 2.44 & 1.35e-04 & 2.50 \\
& 128 & 7.18e-04 & 2.47 & 2.34e-05 & 2.47 & 2.38e-05 & 2.50 \\
\bottomrule
\end{tabular}
\end{table}

\begin{table}[ht]
\normalsize
\centering
\caption{$L^2$-norm error and convergence rate in time.}
\label{table:time:L2}
\begin{tabular}{cccccccccccccc}
\toprule
&\multicolumn{2}{c}{S-mesh} &\multicolumn{2}{c}{BS-mesh}
&\multicolumn{2}{c}{B-type mesh}\\
\cmidrule(r){2-3} \cmidrule(r){4-5} \cmidrule(r){6-7}
$\Delta t$& $L^2$-error & $r_2$ & $L^2$-error & $r_2$ & $L^2$-error & $r_2$
\\
\midrule
0.5 & 7.35e-03 & - & 7.35e-03 & - & 7.35e-03 & - \\
0.25 & 1.80e-03 & 2.03 & 1.80e-03 & 2.03 & 1.80e-03 & 2.03\\
0.125 & 4.53e-04 & 1.99 & 4.53e-04 & 1.99 & 4.53e-04 & 1.99\\
0.0625 & 1.13e-04 & 2.00 & 1.13e-04 & 2.00 & 1.13e-04 & 2.00\\
\bottomrule
\end{tabular}
\end{table}

\begin{table}[ht]
\normalsize
\centering
\caption{Energy-norm error and convergence rate in time.}
\label{table:time:energy}
\begin{tabular}{cccccccccccccc}
\toprule
&\multicolumn{2}{c}{S-mesh} &\multicolumn{2}{c}{BS-mesh}
&\multicolumn{2}{c}{B-type mesh}\\
\cmidrule(r){2-3} \cmidrule(r){4-5} \cmidrule(r){6-7}
$\Delta t$& energy-error & $r_2$ & energy-error & $r_2$ & energy-error  & $r_2$
\\
\midrule
0.5 & 7.35e-03 & - & 7.35e-03 & - & 7.35e-03 & - \\
0.25 & 1.85e-03 & 1.99 & 1.85e-03 & 1.99 & 1.85e-03 & 1.99 \\
0.125 & 4.63e-04 & 2.00 & 4.62e-04 & 2.00 & 4.62e-04 & 2.00 \\
0.0625 & 1.21e-04 & 1.94 & 1.15e-04 & 2.00 & 1.15e-04 & 2.00 \\
\bottomrule
\end{tabular}
\end{table}

\begin{table}[htp]
\normalsize
\centering
\caption{ $L^2$-norm error and energy-norm error
 for different $\varepsilon$.}
\label{table:eps}
\begin{tabular}{ccccccc}
\toprule
 & \multicolumn{3}{c}{$L^2$-error} &\multicolumn{3}{c}{energy-error}\\
\cmidrule(r){2-4} \cmidrule(r){5-7}
$\varepsilon$ & S-mesh & BS-mesh & B-type mesh & S-mesh & BS-mesh & B-type mesh\\
\midrule
$10^{-4}$  & 1.91e-04 & 1.86e-04 & 1.85e-04
& 1.55e-02 & 2.73e-03 & 2.74e-03\\
$10^{-5}$  & 1.87e-04 & 1.86e-04 & 1.86e-04
& 1.55e-02 & 2.73e-03 & 2.75e-03\\
$10^{-6}$  & 1.86e-04 & 1.86e-04 & 1.86e-04
& 1.55e-02 & 2.73e-03 & 2.75e-03\\
$10^{-7}$  & 1.86e-04 & 1.86e-04 & 1.86e-04
& 1.55e-02 & 2.73e-03 & 2.75e-03\\
$10^{-8}$  & 1.86e-04 & 1.86e-04 & 1.86e-04
& 1.55e-02 & 2.73e-03 & 2.75e-03\\
$10^{-9}$  & 1.86e-04 & 1.86e-04 & 1.86e-04
& 1.55e-02 & 2.73e-03 & 2.75e-03\\
$10^{-10}$  & 1.86e-04 & 1.86e-04 & 1.86e-04
& 1.55e-02 & 2.73e-03 & 2.75e-03\\
$10^{-11}$  & 1.86e-04 & 1.86e-04 & 1.86e-04
& 1.55e-02 & 2.73e-03 & 2.74e-03\\
\bottomrule
\end{tabular}
\end{table}

%

\section*{Appendix}
\setcounter{equation}{0}
\setcounter{subsection}{0}
\renewcommand{\theequation}{A.\arabic{equation}}
\renewcommand{\thesubsection}{A.\arabic{subsection}}

In the Appendix,
we sketch some improved results in the one-dimensional case in space
and comment the alternative of a DG discretization in time.

\subsection{The one-dimensional case}
\setcounter{equation}{0}

We examine the LDG method for the one-dimensional singularly perturbed problem
\begin{subequations}\label{cd:spp}
\begin{align}
& u_t-\varepsilon u_{xx} + a(x) u_x + b(x) u
= f(x,t),  & \textrm{in} & \;  (0,1)\times (0,T],
\\
& u(x,0)=u_0(x),& \textrm{in} & \;  [0,1],
\\
& u(0,t) =u(1,t) = 0, & \textrm{for} & \;t\in (0,T].
\end{align}
\end{subequations}
Assume that $u$ can be decomposed as  (see \cite{Roos2008})
\[
u(x,t) = S(x,t)+E(x,t),
\]
with
\begin{align}\label{reg:u:1d}
\left| \partial_x^{j}\partial_t^{m}S(x,t)\right|
\leq C, \quad
\left| \partial_x^{j}\partial_t^{m} E(x,t)\right|
\leq C\varepsilon^{-j} e^{-\alpha(1-x)/\varepsilon}
\quad \forall j,m.
\end{align}

Define the discontinuous finite element space
\[
\mathcal{V}_{N}= \{v\in L^2 (\Omega)\colon
v|_{I_{j}} \in \mathcal{P}^{k} (I_{j}), j=1,...N \},
\]
where $\mathcal{P}^{k} (I_{j})$
denotes the space of polynomials
in $I_{j}$ of degree  $k\geq 0$.

The fully-discrete LDG $\theta$-scheme reads:\\
Let $\uph^0 = \pi u_0$ be the one-dimensional
$L^2$ projection of $u_0$.
For any $m=1,2,\dots,M$,
find the numerical solution
$\wph^{m}=(\uph^m,\qph^m)\in \mathcal{V}_N^2$ such that
\begin{align}\label{fully:theta:scheme:1d}
\dual{\frac{\uph^m-\uph^{m-1}}{\Delta t}}{\vphu}+B(\wph^{m,\theta};\vph)
=\dual{f^{m,\theta}}{\vphu},
\end{align}
holds for any $\vph=(\vphu,\vphq)\in \mathcal{V}_N^2$.
Here $B(\cdot;\cdot)$ is defined analogously as in the two-dimensional case.
The related energy norm is
\begin{align}
\label{energy:norm}
\enorm{\wN}^2
=
\varepsilon^{-1}\norm{\qN}^2
+\norm{(b-a_x/2)^{1/2}\uN}^2
+\sum_{j=0}^{N}\frac12 a_j\jump{\uN}^2_{j}
+\lambda_N \jump{\uN}_N^2.
\end{align}

A careful analysis shows that in 1D  an improved error estimate is possible:

\begin{theorem}\label{thm:theta:scheme}
The errors of fully-discrete LDG $\theta$-scheme
\eqref{fully:theta:scheme:1d} satisfy
\begin{subequations}\label{error:estimate:1d}
\begin{align}
\norm{u(T)-\uph^M}
\leq &\; C(1+T)\big((\Delta t)^\nu+(N^{-1}\max|\psi^{\prime}|)^{k+1}\big),
\\
\Delta t\sum_{m=1}^M \enorm{(\bm w-\wph)^{m,\theta}}
\leq &\; C(1+T)\big((\Delta t)^\nu+(N^{-1}\max|\psi^{\prime}|)^{k+1/2}\big),
\end{align}
\end{subequations}
where $\nu=2$ for $\theta=1/2$ and $\nu=1$ for $1/2<\theta\leq 1$.
Here $C>0$ is a constant independent of $\varepsilon$ and $N$.
\end{theorem}

The proof is based on two improved approximation error estimates for the Ritz projection:
\begin{subequations}\label{app:Ritz:projection:1d}
\begin{align}
\label{app:ritz:u:q}
\enorm{\bm w-\digamma \bm w}\leq &\; C(N^{-1}\max|\psi^{\prime}|)^{k+1/2},
\\
\label{app:ritz:ut}
\norm{u_t-\digamma_1 u_t}\leq &\; C(N^{-1}\max|\psi^{\prime}|)^{k+1}.
\end{align}
\end{subequations}
For that orthogonality and the exact collocations of
the two one-dimensional Gauss-Radau projections $\pi^\pm$ used are important:
\begin{subequations}\label{GR:projection}
\begin{align}
\begin{cases}
\dual{\pi^{-} z}{\vphu}_{I_j} = \dual{z}{\vphu}_{I_j},
\; \forall \vphu\in \mathcal{P}^{k-1}(I_j),
\\
(\pi^- z)^-_{j}= z^{-}_{j}.
\end{cases}
\\
\begin{cases}
\dual{\pi^{-} z}{\vphu}_{I_j} = \dual{z}{\vphu}_{I_j},
\; \forall \vphu\in \mathcal{P}^{k-1}(I_j),
\\
(\pi^+ z)^+_{j-1}= z^{+}_{j-1}.
\end{cases}
\end{align}
\end{subequations}

\subsection{DG time-discretization}
\label{sec:DGLDG}

Instead of the $\theta$-scheme one can also
use a DG time-discretization. We present a corresponding result for the one-dimensional problem in space.

Let $r\geq 0$ be the polynomial order of the elements in time, and the discrete function space $\mathcal{V}_N^{\Delta t}$
 be given by
\begin{align}
\mathcal{V}_N^{\Delta t} \equiv
\{V\in L^2(0,T; \mathcal{V}_N):
V|_{K^m}\in \mathcal{P}^r(K^m,\mathcal{V}_N), m=1,2,\dots,M\}.
\end{align}
Here $\mathcal{P}^r(K^m,\mathcal{V}_N)$ denotes the space of $\mathcal{V}_N$-valued polynomials
of degree $r$ on the time interval $K^m$.
For a function $V\in \mathcal{V}_N^{\Delta t}$, we define at $t=t^m$ the one-sided limits
and the jump by $V^{m,\pm}= \lim_{t\rightarrow t^m\pm 0} V(t)$, $\jump{V}^m = V^{m,+}-V^{m,-}$.

Then, the space-time DG scheme reads:\\
 Find the discrete solution
$\wph=(\uph,\qph)\in \mathcal{V}_N^{\Delta t}\times \mathcal{V}_N^{\Delta t}$ such that
\begin{align}
\label{fully:DG:scheme}
\int_{K_m}(\dual{\uph_t}{\Vphu}+B(\wph;\Vph)) dt
+\dual{\jump{\uph}^{m-1}}{\Vphu^{m-1,+})}
=\int_{K_m} \dual{f}{\Vphu}dt
\end{align}
holds for any test function
$\Vph=(\Vphu,\Vphq)\in \mathcal{V}_N^{\Delta t}\times \mathcal{V}_N^{\Delta t}$,
$m=1,2,\dots,M$.
For the initial value we take $\uph^{0,-}=\pi u_0$ as before.

Following the ideas of \cite{Franz2018,Vlasak2014}
and the estimation \eqref{app:Ritz:projection:1d} for Ritz projection,
we can establish the following error estimates:

\begin{theorem}
Let $\bm w=(u,q)=(u,\varepsilon u_x)$ where $u$ is the solution of \eqref{cd:spp} satisfying \eqref{reg:u:1d}.
Let $\wph=(\uph,\qph)\in \mathcal{V}_N^{\Delta t}\times \mathcal{V}_N^{\Delta t}$
be the numerical solution given by \eqref{fully:DG:scheme}. Then the errors satisfy
\begin{subequations}
\begin{align}
\norm{u(T)-\uph(T)} \leq &\; C\big((\Delta t)^{r+1}+(N^{-1}\max|\psi^{\prime}|)^{k+1}\big),
\\
\left(\sum_{m=1}^M \int_{K^m}\enorm{\bm w-\wph}^2 dt\right)^{1/2}
\leq &\; C\big((\Delta t)^{r+1}+(N^{-1}\max|\psi^{\prime}|)^{k+1/2}\big),
\end{align}
\end{subequations}
where $C>0$ is a constant independent of $\varepsilon$ and $N$.
\end{theorem}

\section*{Acknowledgements}

This research was supported by National Natural Science Foundation of China
(No. 11801396),
Natural Science Foundation of Jiangsu Province
(No. BK20170374),
Natural Science Foundation of the Jiangsu Higher Education Institutions of China
(No. 17KJB110016).

\end{document}